\documentclass[11pt]{article}
\usepackage{amsfonts}

\usepackage{graphics}
\usepackage{indentfirst}
\usepackage{cite}
\usepackage{latexsym}
\usepackage{amsmath,amsthm}
\usepackage{amssymb}
\usepackage{amscd}
\usepackage{mathrsfs}

\usepackage{color}

\newtheorem{theorem}{Theorem}[section]
\newtheorem{remark}{Remark}[section]

\newtheorem{definition}{Definition}[section]
\newtheorem{lemma}[theorem]{Lemma}

\newcommand{\n}{\rho}

\newcommand{\mr}{\mathbb{R}}
\newcommand{\lm}{\lambda}

\renewcommand{\div}{ {\rm div }  }

\newcommand{\na}{\nabla }

\newcommand{\pa}{\partial}

\newcommand{\bt}{\begin{theorem}}
\newcommand{\bl}{\begin{lemma}}
\newcommand{\el}{\end{lemma}}
\newcommand{\et}{\end{theorem}}
\newcommand{\ga}{\gamma}

\newcommand{\de}{\delta}

\newcommand{\la}{\label}
\newcommand{\si}{\sigma}

\newcommand{\ol}{\overline}

\newcommand{\bn}{\begin{eqnarray}}
\newcommand{\en}{\end{eqnarray}}
\newcommand{\bnn}{\begin{eqnarray*}}
\newcommand{\enn}{\end{eqnarray*}}

\newcommand{\bnnn}{\begin{eqnarray*}}
\newcommand{\ennn}{\end{eqnarray*}}
\newcommand{\ben}{\begin{enumerate}}
\newcommand{\een}{\end{enumerate}}

\newcommand{\rs}{\rho^*}

\newcommand{\pb}{\ol{P}}
\newcommand{\T}{\mathbb{T}}

\newcommand{\du}{\dot{u}}

\newcommand{\ba}{\begin{aligned}}
\newcommand{\ea}{\end{aligned}}
\newcommand{\be}{\begin{equation}}
\newcommand{\ee}{\end{equation}}

\def\p{\partial}
\def\norm[#1]#2{\|#2\|_{#1}}

\def\lam{\lambda}

\def\o{\omega}

\makeatletter      
\@addtoreset{equation}{section}
\makeatother       

\title{Global Existence and Incompressible Limit for
Compressible Navier-Stokes Equations with Large Bulk Viscosity Coefficient 
and Large Initial Data}

\date{}

\author{$\text{Qinghao L{\small EI}}^{a,b}, \text{Chengfeng X{\small{IONG}}}^{a,b}\thanks{Email addresses:  leiqinghao22@mails.ucas.ac.cn (Q. H. Lei), xiongchengfeng20@mails.ucas.ac.cn (C. F. Xiong) }$\\
    a. School of Mathematical Sciences,\\ University of Chinese Academy of Sciences,
     Beijing 100190, P. R. China;\\
    b. Institute of Applied Mathematics,\\ Academy of Mathematics and Systems Science (AMSS),\\
     Chinese Academy of Sciences, Beijing 100190, P.R. China.}

\begin{document}
\maketitle

\begin{abstract}
For periodic initial data with the density allowing vacuum, 
we establish the global existence and exponential decay 
of weak, strong and classical solutions
to the two-dimensional(2D) compressible Navier-Stokes equations
when the bulk viscosity coefficient is sufficiently large, 
without any extra restrictions on initial velocity divergence.
Moreover, we demonstrate that when the bulk viscosity coefficient tends to infinity,
these solutions converge to solutions to 
 the inhomogeneous incompressible Navier-Stokes equations.
For the incompressible limit of weak solutions, 
 our results hold
even without requiring the initial velocity field to be divergence-free.
Our results are established by introducing time-layers 
to avoid 
imposing restrictions on the initial velocity divergence, along with 
estimates of $L^\infty$-norm of the effective viscous flux $G$ via a  time-partitioning
approach based on Gagliardo-Nirenberg's inequality.
\\
\par\textbf{Keywords:} Compressible Navier-Stokes equations; 
Global existence; Incompressible limit; Large initial data; Vacuum
\end{abstract}

\section{Introduction and main results}
We study the two-dimensional barotropic compressible
Navier-Stokes equations which read as follows:
\be\ba\la{ns}
\begin{cases}
  \rho_t + \div(\rho u) = 0,\\
  (\n u)_t + \div(\n u\otimes u) -\mu \Delta u - (\mu + \lm)\na\div u
    +\na P = 0,
\end{cases}
\ea\ee
where $\n=\n(x,t)$ and $u(x,t)=(u^1(x,t),u^2(x,t))$ represent the 
density and velocity of the compressible flow respectively, and the 
pressure $P$ is given by
\be\la{i1}
  P=a\n^\ga,
\ee
with contants $a>0,\ga > 1$. Without loss of generality, it is assumed that
$a=1$. The shear viscosity coefficient $\mu$ 
and the bulk viscosity coefficient $\lam$ satisfy
the physical restrictions:
\be\la{i2}
\mu>0,\quad \mu+\lam\geq 0.
\ee
For later purpose, we set 
\be\la{nu}
\nu := 2\mu + \lam,
\ee
which together with (\ref{i2}) yields that 
\be\la{numu}\nu\geq\mu.\ee
We consider the Cauchy problem supplemented 
with the initial data $\n_0$ and $u_0$, which are periodic with period 1
in each space direction $x_i,\ i=1,2$, i.e. functions defined
on $\mathbb{T}^2 = \mr^2/\mathbb{Z}^2$. We also required that
\be\la{i3}
\n(x,0)=\n_0(x),\quad \n u(x,0)=\n_0u_0(x), \quad x\in \mathbb{T}^2.
\ee
It is obvious that the total mass and momentum of smooth enough solutions
of (\ref{ns}) are conserved through the evolution, that is, for all
$t>0$,
\be\ba\la{conv}
\int_{\mathbb{T}^2} \n dx = \int_{\mathbb{T}^2}\n_0 dx,
\quad \int_{\mathbb{T}^2} \n u dx = \int_{\mathbb{T}^2} \n_0 u_0 dx.
\ea\ee
Without loss of generality, we shall assume that 
\be\ba\la{asm}
\int_{\T^2} \n_0 dx = 1, \quad
\int_{\mathbb{T}^2} \n_0 u_0 dx = 0.
\ea\ee

There is a vast amount of literature on the global existence and 
large time behavior of solutions to (\ref{ns}). The one-dimensional problem has 
been extensively investigated by many researchers (see \cite{H4,KS,S1,S2} 
and references therein).
 As for the multi-dimensional case, the local existence
and uniqueness of classical solutions were established in \cite{N,S} 
in the absence of vacuum. For strong solutions, 
results were obtained in \cite{CCK,CK,CK2,SS,LLL}(and references therein),  
where the initial density does not necessarily have to be positive 
and can vanish on open sets.
The first result concerning global classical solutions was established by
Matsumura-Nishida\cite{MN1} where initial data are required to be close
to a non-vacuum equilibrium in some Sobolev space $H^s$.
Therefore, the classical solution represents a small perturbation from a 
non-vacuum state, ensuring that the density is strictly away from vacuum 
and the gradient of density remains bounded uniformly in time. 
Thereafter, Hoff \cite{H1,H2,H3} studied the problem for discontinuous
initial data. 
For the existence of weak solution,
the major pioneering breakthrough was made by Lions \cite{L2} 
 (see also Feireisl-Novotn\'y-Petzeltov\'a \cite{FNP} and Feireisl \cite{F}),
where he established the global existence of finite-energy weak solutions
with the pressure $P = a\n^\ga(a>0,\ga>1)$ , 
provided that $\ga $ is suitably large.
The main restriction is the initial data with finite energy 
and the presence of vacuum in the initial density.
The large time behavior of the finite-energy weak solution was obtained
in \cite{NS,FP} and more recently, Peng-Shi-Wu \cite{PSW} obtained the 
exponential decay of finite-energy weak solutions.
It is noteworthy that Huang-Li-Xin\cite{HLX2} established the global
existence and uniqueness of classical solutions to the Cauchy problem
in three-dimensional space,
 under the assumption of smooth initial data with
suitably small energy but potentially large oscillations. 
In particular,
the initial density may contain vacuum, even has compact support.
Subsequently, Li-Xin\cite{LX2} extended this result to the Cauchy 
problem in $\mathbb{R}^2$ while Huang-Li \cite{HL} generalized this result
to the three-dimensional Full Navier-Stokes equations.
More recently, Danchin-Mucha \cite{DM} demonstrated the existence of 
global weak solutions with uniformly bounded density in time
under the assumption that the bulk viscosity is sufficiently large and that  $\nu^{1/2} \|\div u_0\|_{L^2}$ has scale restrictions.
Additionally, for the isothermal case $\ga = 1$, the uniqueness was 
also established.

The aim of this paper is to generalize the global existence result
in \cite{DM} to remove the scale restrictions on $\nu^{1/2} \| \div u_0\|_{L^2}$.
It is worth noting that by removing this restriction, 
when obtaining the singular limit of the weak solutions
from the  compressible Navier-Stokes equations 
to the incompressible one,
we can assume that the initial velocity field is not divergence-free.
Motivated by the blow up criteria 
established in \cite{HLX1}, we are also able to extend the global existence result to the strong and classical solutions.
Moreover, similar to \cite{CL,Ge,LX,HL2}, 
the uniqueness and large time behavior with exponential decay will be  
established.

Before stating the main results, we first explain the notations
and conventions used throughout this paper. We denote
\be\ba\la{i4}
  \int f dx = \int_{\mathbb{T}^2} fdx,\quad
   \ol{f}=\frac{1}{|\mathbb{T}^2|}\int fdx.
\ea\ee
For $1\leq r\leq\infty$, we also denote the standard Lebesgue and Sobolev
spaces as follows:
\be\la{i5}
L^r =L^r(\mathbb{T}^2),\quad W^{s,r} =W^{s,r}(\mathbb{T}^2),
\quad H^s =W^{s,2}.
\ee
The initial total energy is defined as follows:
\be\la{e0}\ba
E_0 := \int \frac{1}{2} \rho_0 |u_0|^2 + \frac{1}{\ga-1}\n^\ga_0 dx.
\ea\ee
Moreover, we further define
\be\la{xd}\ba
\o :=\pa_1 u^2 - \pa_2 u^1.
\ea\ee

Then we provide the definition of weak and strong solutions to (\ref{ns}).
\begin{definition}
If $(\n,u)$ satisfies \eqref{ns} in the sense of distribution, then we call $(\n,u)$ a weak solution.
Moreover, for a weak solution if
all derivatives involved in \eqref{ns} are regular distributions
and equations \eqref{ns} hold almost everywhere in 
$\mathbb{T}^2\times(0,T)$, then $(\n,u)$ is called a strong solution.
\end{definition}
The first main result concerning the global existence and exponential decay of weak  solutions can be described as follows:
\begin{theorem}\la{th0}
Assume the initial data $(\n_0,u_0)$ satisfy
\be \la{wsol1}\ba
0\le \n_0 \in L^\infty,\quad u_0 \in H^1.
\ea\ee
Then, there exists a positive constant $\nu_1$ depending only on 
$\ga,\ \mu,\ E_0,\ \|\n_0\|_{L^1\cap L^\infty}$ and $\| \na u_0\|_{L^2}$, such that when $\nu \ge \nu_1 $, 
the problem \eqref{ns}--\eqref{i3} has at least one weak solution $(\n,u) \in \T^2 \times (0,\infty)$ satisfying,
for any $0<T<\infty$ and $1 \le p < \infty $
\be\la{wsol2}\ba
0\le \n(x,t) \leq 2 \tilde{\n},
\quad \mathrm{for\ any\ }(x,t)\in \T^2 \times[0,\infty),
\ea\ee
\be\la{wsol3}\ba
\begin{cases}
\rho\in L^{\infty}(\T^2 \times (0,\infty)) \cap C([0,\infty);L^p), \\ 
u\in L^\infty(0,T;H^1), t^{1/2}u_t \in L^2(0,T;L^2), t^{1/2} \na u \in L^\infty(0,T;L^p),
\end{cases}
\ea\ee
where $\tilde{\n}=\|\n_0\|_{L^\infty} + \left( (\ga-1)E_0 \right)^{1/\ga}$.

Moreover, when $\nu \ge \nu_1$, for any $s\in [1,\infty)$, 
there exist positive constants $C$ and $K_0$, where $C$ depends only on
$s,\ \mu,\ E_0,\ \ga,\ \| \rho_0 \|_{L^1\cap L^\infty}$, 
while $K_0$ depends only on $\mu,\ \ga, \ \| \rho_0 \|_{L^1\cap L^\infty}$,
such that for  $\alpha_0 =\frac{K_0}{\nu}$, 
it holds that for any $t\ge 1$,
\be\la{wsol4}\ba
\| \n-\ol{\n_0} \|^{s}_{L^s} \le C e^{-2 \alpha_0 t},\quad
\| \o\|^2_{L^2} +\nu \| \div u\|^2_{L^2}  \le C e^{-\alpha_0 t},\quad  \| \sqrt{\n} \dot u \|^2_{L^2} \le C e^{-\alpha_0 t}.
\ea\ee  
\end{theorem}
\begin{theorem}\la{th01}
Fix the initial data $(\n_0,u_0)$ in $L^\infty\times H^1$ satisfying $\n_0\geq 0$.
Assume that $\nu_1$ is the one determined in Theorem \ref{th0}.
For $\nu \ge \nu_1$, 
we denote  $(\n^{\nu},u^{\nu})$ as the global weak solution to \eqref{ns}--\eqref{i3} established in Theorem \ref{th0}.
Then, as $\nu$ tends to $\infty $, the solution sequence $(\n^{\nu},u^{\nu})$ 
admits a subsequence that converges to the global solution of the following 
inhomogeneous incompressible Navier-Stokes equations:
\be\la{isol2}\ba
\begin{cases}
\n_t+\div(\n u)=0,\\
(\n u)_t+\div(\n u\otimes u) -\mu \Delta u + \na \pi =0, \\
\div u=0,
\end{cases}
\ea\ee
with initial data $\n(\cdot,0)=\n_0,\ \n u(\cdot,0)=m_0:=\n_0 u_0$,
and $(\n,u)$ satisfies for any $0<T<\infty$ and $1\le p <\infty$,
\be\la{isol3}\ba
\begin{cases}
\rho\in L^{\infty}(\T^2 \times (0,\infty)) \cap C([0,\infty);L^p), \\ 
u\in L^2(0,T;H^1), \quad \na u \in L^2(\T^2 \times (0,\infty)), \\
\sqrt{t} \na^2 u,\ t \sqrt{\n} u_t,\ t^2 \na u_t \in L^2(0,T;L^2), \\
\sqrt{t} \na u,\ t \na^2 u,\ t^2 \sqrt{\n} u_t \in L^\infty(0,T;L^2), \\
\sqrt{t} \pi \in L^2(0,T;H^1), \quad t \pi \in L^\infty(0,T;H^1),
\end{cases} 
\ea\ee
and we also have for any $0 < \tau <\infty$,
\be\la{isol4}\ba
\div u^{\nu} = O(\nu^{-1/2}) \  in \  L^2(\T^2 \times (0,\infty)) \cap L^\infty(\tau,\infty;L^2).
\ea\ee
If the initial data $(\n_0,u_0)$ further satisfy
\be\la{ws}\ba
0\le \n_0 \in L^\infty,\quad u_0 \in H^1,\quad \div u_0=0,
\ea\ee
then the entire sequence $(\n^{\nu},u^{\nu})$ converges to the unique global solution of \eqref{isol2},
and $(\n,u)$ satisfies for any $0<T<\infty$ and $1\le p <\infty$,
\be\la{lws1}\ba
\begin{cases}
\rho\in L^{\infty}(\T^2 \times (0,\infty)) \cap C([0,\infty);L^p),
\quad \sqrt{\n} u \in C([0,\infty);L^2), \\ 
u\in L^\infty(0,T;H^1), \quad \na u \in L^\infty(0,\infty;L^2), \\
\na u,\ \sqrt{\n} u_t,\ \na^2 u,\ \na \pi \in L^2(\T^2 \times (0,\infty)), \\ 
\sqrt{t} \na \pi,\ \sqrt{t} \na^2 u \in L^\infty(0,\infty;L^2) \cap L^2(0,T; L^p), \\
\sqrt{\n} u_t \in L^\infty(0,T;L^2) \cap L^2(0,T; L^p), \quad 
\sqrt{t} \na u_t \in L^2(\T^2 \times (0,T)),
\end{cases} 
\ea\ee
and
\be\la{lws2}\ba
\div u^{\nu} = O(\nu^{-1/2}) \  in \  L^2(\T^2 \times (0,\infty)) \cap L^\infty(0,\infty;L^2).
\ea\ee
If the initial data $(\n_0,u_0)$ additionally satisfy
\be\la{insc1}\ba
0\le \n_0 \in H^2,\quad u_0 \in H^2 ,\quad \div u_0=0,
\ea\ee 
and the following compatibility condition:
\be\la{insc2}\ba
- \mu \Delta u_0 + \nabla \pi_0 = \sqrt{\n_0}g_1, 
\ea\ee
for some $(\pi_0,g_1)\in H^1 \times L^2$,
then the limit of $(\n^\nu,u^\nu)$ is the unique global strong solution of \eqref{isol2} and satisfies
\be\la{insc3}\ba
\begin{cases}
\n \in C( [0,\infty );H^{2} ), \quad  
u \in C( [0,\infty );H^{2} ) \cap L^2(0,T;H^3), \\ 
\pi \in C( [0,\infty );H^1 ) \cap L^2(0,T;H^2), \\
u_t \in L^2(0,T;H^1), \quad
(\n_t, \sqrt{\n} u_t) \in L^\infty(0,T;L^2),
\end{cases} 
\ea\ee
for any $0<T<\infty$.
\end{theorem}

\begin{remark}\la{ctoi1}
In contrast to \cite{DM}, we can establish the convergence of solutions for the compressible Navier-Stokes equations to those of the incompressible Navier-Stokes equations without requiring the assumption ${\rm div} u_0 \equiv 0$.
Moreover, it is shown in \cite[Theorem 2.1]{L1} that there exist global weak solutions for (\ref{isol2}) when $ \div u_0 \ne 0$.
\end{remark}

\begin{remark}\la{ctoi2}
For the solution $(\n,u)$ of \eqref{isol2} satisfying \eqref{isol3} with the initial data $(\n_0,m_0) = (\n_0,\n_0u_0)$,
meaning that for all $1\leq p <\infty$, $\n\in C([0,\infty);L^p)$ and 
with $\n(\cdot,0)=\n_0$ and moreover,
for all $T\in (0,\infty)$, 
$(-\Delta)^{-1/2}\na^\bot\cdot \n u \in C([0,T];L^2_w)$ with 
$(-\Delta)^{-1/2}\na^\bot\cdot \n u(\cdot,0) = (-\Delta)^{-1/2}\na^\bot\cdot m_0$.
The reason why we cannot  obtain the time-continuitiy of $\n u $ is that 
$ \pi $ is only in $L^2(\tau,\infty;H^1)$ rather than
$L^2(0,\infty;H^1)$. We refer readers to \cite[Theorem 2.2]{L1} for more discussions about the time-continuity of 
$\n u$.
\end{remark}

\begin{theorem}\la{th1}
Suppose that the initial data $(\n_0,u_0)$ satisfy for some $q>2$,
\be \la{ssol1}\ba
0\le \n_0 \in W^{1,q},\quad u_0 \in H^1.
\ea\ee 
Then, for the same $\nu_1$ in Theorem \ref{th0}, when $\nu \ge \nu_1 $, 
the problem \eqref{ns}--\eqref{i3} has a unique strong solution $(\n,u)$ in $\T^2 \times (0,\infty)$ 
satisfying \eqref{wsol2}
and 
\be\la{ssol4}\ba
\begin{cases}
\rho\in C([0,T];W^{1,q} ), \quad \n_t\in L^\infty(0,T;L^2), \\ 
u\in L^\infty(0,T; H^1) \cap L^{(q+1)/q}(0,T; W^{2,q}), \\ 
t^{1/2}u \in L^2(0,T; W^{2,q}) \cap L^\infty(0,T;H^2), \\
t^{1/2}u_t \in L^2(0,T;H^1), \\
\n u\in C([0,T];L^2), \quad \sqrt{\n} u_t\in L^2(\T^2 \times(0,T)),
\end{cases} 
\ea\ee
for any $0<T<\infty$.
Moreover, the strong solution $(\n,u)$ satisfies \eqref{wsol4}.
\end{theorem}
Next, if the initial data $(\n_0,u_0)$ satisfy higher regularity and compatibility conditions, 
we can obtain the global existence and exponential decay of classical solutions.
\begin{theorem}\la{th2}
Assume that the initial data $(\n_0,u_0)$ satisfy for some $q>2$,
\be\la{csol1}\ba
0\le \n_0 \in W^{2,q},\quad u_0 \in H^2,
\ea\ee 
and the following compatibility condition:
\be \la{csol2} \ba
- \mu \Delta u_0 - (\mu + \lm )\nabla \div u_0 +  \nabla P(\n_0)=\n_0^{1/2}g_2,
\ea\ee
for some $g_2 \in L^2.$
Then, for the value $\nu_1$ determined in Theorem $\ref{th0}$, when $\nu \ge \nu_1 $, 
the problem $(\ref{ns})-(\ref{i3})$ has a unique classical solution
 $(\n,u)$ in $\T^2 \times (0,\infty)$ satisfying \eqref{wsol2}
and 
\be\la{csol4}\ba
\begin{cases}
(\rho,P(\n))\in C([0,T];W^{2,q} ),\quad  (\n_t,P_t)\in L^\infty(0,T;H^1), \\ 
(\n_{tt},P_{tt})\in L^2(0,T;L^2), \\
u\in L^\infty(0,T; H^2) \cap L^2(0,T;H^3), \quad u_t \in L^2(0,T; H^1) \\ 
\na u_t, \na^3u \in L^{(q+1)/q}(0,T;L^q), \\
t^{1/2}\na^3u \in L^\infty(0,T;L^2)\cap L^2(0,T; L^q), \\
t^{1/2}u_t\in L^\infty(0,T;H^1) \cap L^2(0,T;H^2), \\ 
t^{1/2}\na^2(\n u)\in L^\infty(0,T;L^q),\quad \n^{1/2}u_t\in L^{\infty}(0,T;L^2), \\
t\n^{1/2}u_{tt},\quad t\na^2 u_t \in L^{\infty}(0,T;L^2), \\
t\na^3 u \in L^{\infty}(0,T;L^q),\quad t\na u_{tt} \in L^2(0,T;L^2),
\end{cases} 
\ea\ee
for any $0<T<\infty$.
Furthermore, the classical solution $(\n,u)$ satisfies \eqref{wsol4}.
\end{theorem}

Finally, following the approach in \cite{CL,LX} and leveraging the exponential decay estimates $(\ref{wsol4})$,
we can establish the following large-time behavior of the spatial gradient of the density for the strong solution in Theorem $\ref{th1}$ when vacuum
appears initially.
\begin{theorem}\la{th3}
In addition to the assumptions in Theorem $\ref{th1}$, 
we further assume that there exists some point $x_0\in \T^2$ satisfying $\n_0(x_0)=0$. 
Then for any $r>2$, there exists a positive constant $C$ depending only on $r,\ \mu,\ E_0,\ \ga,\ \| \rho_0 \|_{L^1\cap L^\infty}$, such that for any $t \ge 1$
\be\la{pbu0}\ba
\| \na \n(\cdot ,t) \|_{L^r} \ge C e^{\alpha_0 \frac{r-2}{r} t }.
\ea\ee	
\end{theorem}

\begin{remark}\la{lrk1}
Combining the condition $q>2$ with (\ref{csol4}), we deduce that
\be\la{csol5}\ba
\n, P(\n) \in C([0,T];W^{2,q})\hookrightarrow C\left([0,T];C^1(\T^2) \right).
\ea\ee
Moreover, the standard embedding and (\ref{csol4}) yield for any $0<\tau<T<\infty$
\be\la{csol6}\ba
u\in L^\infty(\tau ,T;W^{3,q})\cap H^1(\tau ,T;H^2)\hookrightarrow
C\left([\tau ,T];C^2(\T^2) \right),
\ea\ee
and	
\be\la{csol7}\ba
u_t\in L^\infty(\tau ,T;H^2)\cap H^1(\tau ,T;H^1)\hookrightarrow
C\left([\tau ,T]; C(\T^2) \right).
\ea\ee
In view of (\ref{ns}), (\ref{csol5}) and (\ref{csol6}), we have
\be\la{csol8}\ba
\n_t = -\n \div u - u \cdot \na \n \in C(\T^2 \times [\tau,T]).
\ea\ee
Together with (\ref{csol5}), (\ref{csol6}) and (\ref{csol7}), we conclude that
the solution in Theorem \ref{th2} 
is in fact a classical solution to system \eqref{ns}--\eqref{i3} in $\T^2 \times (0,\infty)$.
\end{remark}

\begin{remark}\la{lrk02}
In our preceding analysis, we may assume without loss of generality that the initial data satisfy the condition
\bnn
\int \n_0 u_0 dx = 0.
\enn
If this condition fails to hold, we perform the following Galilean transformation:
First, we define
\bnn
v_0 = \left( \int \n_0 dx \right)^{-1} \int \n_0 u_0 dx,
\enn
and then set
\bnn
( \check{\n}(x,t),\check{u}(x,t) ) := ( \n(x+v_0t,t),u(x+v_0t,t)-v_0 ).
\enn
A direct calculation shows that the transformed solution satisfies the desired condition:
\bnn
\int \check{\n}(x,0) \check{u}(x,0) dx = 0.
\enn
\end{remark}

\begin{remark}\la{lrk2}
It is worth noting that in our existence Theorem \ref{th0}, we only require the bulk viscosity coefficient to be sufficiently large.
This represents a improvement over the work of Danchin-Mucha \cite[Theorem 2.1]{DM}, in which the initial data are also required to satisfy $\| \div u_0 \|_{L^2} \le K \nu^{-1/2}$ for some $K>0$.
Our result therefore provides a direct generalization of their existence theory.
\end{remark}

We now make some comments on the analysis of this paper. Note that 
for initial data (\ref{csol1})--(\ref{csol2}), the local existence and uniqueness
of classical solutions to the problem (\ref{ns})--(\ref{i3}) have been established 
in a manner similar to \cite{LLL}.
Therefore, to extend the classical solution globally in time, it is 
essential to obtain global a priori estimates for smooth solutions to 
(\ref{ns})--(\ref{i3}) in appropriate higher norms. 

The key issue addressed in this paper is to obtain both the 
time-independent upper bound of the density and the time-dependent higher
norm estimates of the solution $(\n, u)$. We start with the basic
energy estimate (see (\ref{0x1})) and the $L^2(0,T;L^2)$-norm of the pressure(see (\ref{0xx2})).
Then, combined with the initial layer analysis, these estimates imply that the
$L^\infty(0,T;L^2)$-norm of the velocity gradient 
exhibits a singularity of the order $O(t^{\frac{1}{2}})$ near $t = 0$.
This, in turn, leads to the desired estimate of the 
$L^1(0,\min\{1,T\};L^\infty)$-norm of the effective
viscous flux (see (\ref{gw}) for the definition).
On the other hand, we employ the method in \cite{H1,LZ} to obtain the 
estimate of the material derivative of the velocity.
We derive that the $L^\infty(1,T;L^2)$ norm of $\sqrt{\n} \du$
and the $L^2(1,T;L^2)$ norm of $\na \du$ are uniform with respect to $\nu$ (see (\ref{pd11})).
Subsequently, following the approach in \cite{CL},
we obtain the exponential decay of the density,
the velocity gradient and the material derivative of  the velocity
(see (\ref{pd301})--(\ref{pd303})).
Based on this decay arguments, we derive the time-independent estimates on the 
$L^{1}(\min\{1,T\},T;L^\infty)$-norm
of the effective viscous flux (see (\ref{pd57})).
Then, motivated by \cite{LX}, 
we apply the Zlotnik inequality (see Lemma \ref{zli})
to slightly pull back the upper bound of the density 
(see (\ref{pd56}) and (\ref{pd511})).

Next, similar to the arguments in \cite{HLX3,HLX2,LLL,HLX1},
we proceed to bound the gradients of both density and velocity, 
which is  achieved by a logarithm Gr\"onwall's inequality and 
Beale-Kato-Majda's inequality(see Lemma \ref{bkm}) and the a priori estimates
derived above. Furthermore, these estimates yield a bound on
 the $L^1(0,T;L^\infty)$-norm of the gradient of the velocity 
(see Lemma \ref{s23} and its proof ). 
Finally, using these results,
 we  establish  higher order derivative estimates
of the density and the velocity similar to the arguments in \cite{LLL, HLX3}.

Furthermore, with the $\nu$-independent estimates derived above
(see (\ref{0x1}), (\ref{0x10}), (\ref{pd11}) and (\ref{pd51})), we are able to 
get the singular limit of the compressible Navier-Stokes equations to 
the incompressible one (see the proof of Theorem \ref{th01} in Section 5) 
with the aid of the global well-posedness results established in \cite{DM2,HW}.

Compared to \cite{DM}, one of the major new challenges 
arises from the absence of the restriction $\nu \| \div u_0 \|^2_{L^2}$ on the initial data.
To obtain the singular limit, we are required to establish a priori estimates independent of the viscosity coefficient $\nu$.
However, owing to the absence of the restriction, the upper bounds of the $L^\infty(0,T;L^2)$ norms of $\na u$ and $\sqrt{\nu} \div u$ will be related to $\nu$ (see (\ref{x1000})).
To ensure that this estimate is uniform with respect to $\nu$, we consider introducing the time-layer estimate (see (\ref{0x10})) to avoid the appearance of the initial data.
It should be noted that following the approach in \cite{DM}, we will not be able to introduce the time-layer estimate.
More precisely, when estimating $\na u$, the method employed by \cite{DM} to handle the convective term $\int \n u \cdot \na u \cdot \dot{u} dx$ is as follows:
\be\ba\nonumber
\int \n  u \cdot \na u\cdot \du dx \leq C\left(\int \n |\du|^2 dx\right)^\frac{1}{2} 
\left(\int \n |u|^4 dx\right)^\frac{1}{4}\left(\int |\na u|^4 dx\right)^\frac{1}{4}.
\ea\ee
Then, the second term on the right-hand side is dealt with by applying the logarithmic Ladyzhenskaya inequality.
Nevertheless, the appearance of the logarithmic term is fatal for the introduction of the initial time layer.
Therefore, we have adopted a new method. According to the observations in \cite{HL2}, we decompose $\n \dot{u}$ into $\nu \na \div u+\mu \na^\bot \o - \na P$, and then estimate these terms respectively (see (\ref{0x11})--(\ref{0x20})). 
The fact that this method does not require the logarithmic inequality to estimate $\n \dot{u}$ represents the innovation in establishing the singular limit.

The remainder of this paper is organized as follows: In Section 2, we
will present some known facts and  elementary inequalities  
requisite for later analysis. Section 3 is dedicated to deriving the upper
bound of the density, which is uniform in time and independent of lower
bound of the density and thus, essential for extending the local solution to
all time. Based on the upper bound of the density, a priori estimates of higher-order
derivatives are established in Section 4. Finally, the main results,
Theorems \ref{th0}--\ref{th3} are proved in Section 5, relying on the previous
bounds.

\section{Preliminaries}
In this section, we will recall some known facts and elementary inequalities
which will be used frequently later.

First, we have the following local existence theory of the classical solution, 
and its proof can be found in \cite{LLL}.
\begin{lemma}\la{lct}
Assume $\left( \n _0 ,u_0  \right)$ satisfies that for some $q>2$
\be \la{lct1}\ba
\n _0 \in W^{2,q}, \quad u_0 \in H^2,
\ea \ee 
and the compatibility condition \eqref{csol2}. 
Then there is a small time $T>0$ depending only on
$\mu,\ \lambda,\ \ga,\ q,\ \| \n_0\|_{W^{2,q}},\ \|u_0\|_{H^2}$ and $\| g_2 \|_{L^2}$,
such that there exists a unique classical solution $(\n,u)$ to the problem 
\eqref{ns}--\eqref{i3} in $\T^2 \times (0,T]$ satisfying \eqref{csol4}.

\end{lemma}
Next, the following Gagliardo-Nirenberg's inequalities (see \cite{NI}) will be used frequently later.
\begin{lemma}\la{gn1}
Let $u\in H^1(\T^2)$, there exists a positive constant $C$ depending only on $\T^2$ such that for any $2<p<\infty$
\be\ba\la{gn11}
\| u-\ol{u}\|_{L^p} \le Cp^{1/2}\| u-\ol{u}\|^{2/p}_{L^2} \| \na u\|^{1-2/p}_{L^2}.
\ea\ee
Furthermore, for $1\le r <\infty$, $2<q<\infty$, 
there exists a positive constant $C$ depending only on $r,\ q$
and $\T^2 $ such that for every function $v\in W^{1,q}(\T^2)$ it holds that
\be\ba\la{gn12}
\|v-\ol{v}\|_{L^\infty} \le C\|v-\ol{v}\|^{r(q-2)/2q+r(q-2)}_{L^r} \|\na v\|^{2q/2q+r(q-2)}_{L^q}.
\ea\ee
\end{lemma}

The following Poincar\'e type inequality can be found in \cite{F}.
\begin{lemma}\la{pt}
	Let $v\in H^1$, and let $\n$ be a non-negative function satisfying
\be\ba\nonumber
0<M_1\leq \int \n dx,\quad \int \n^\ga dx \leq M_2,
\ea\ee
with $\ga>1$. Then there exists a positive constant $C$ depending only on 
$M_1,\ M_2$ and $\ga$ such that
\be\ba\la{pt1}
\|v\|_{L^2}^2 \leq C\int \n |v|^2 dx + C \|\na v\|_{L^2}^2.
\ea\ee
\end{lemma}
Next, for $\na^{\bot}:=(-\pa_2,\pa_1)$, denoting the material derivative 
of $f$ by $\frac{D}{Dt}f=\dot{f}:=f_t + u\cdot\na f$,
we now state standard $L^{p}$-estimate for the following elliptic system derived from the momentum equations in (\ref{ns}): 
\be\ba\la{p1}
\Delta G = \div(\n\dot{u}),\quad \mu \Delta \o = \na^{\bot} \cdot(\n \dot{u}),
\ea\ee
with
\be\ba\la{gw}
G:=(2\mu + \lam)\div u - (P-\ol{P}),
\ea\ee
and $\o$ given by (\ref{xd}).

With these notations defined above, we state the following lemma.
\begin{lemma}\la{estg}
Let $(\n,u)$ be a smooth solution of \eqref{ns}. Then for $1<p<\infty$ and positive integer $k \ge 1$, 
there exists a positive constant $C$ depending only on $k,\ p$ and $\mu$, such that 	
\be\la{p2}
  \|\na^k G\|_{L^p}+\|\na^k \o\|_{L^p} \leq C \| \na ^{k-1}(\n \dot{u}) \|_{L^p}.
  \ee
\end{lemma}

Next, the following div-curl estimate will be frequently used in later
arguments.
\begin{lemma}\la{dc}
Let $k \ge 1$ be a positive integer and $1<p<\infty$.
Then there exists a positive constant $C$ depending only on $k$ and $p$
such that for every $\na u\in W^{k,p}$, it holds that 
\be\ba\la{dc1}
\|\na u\|_{W^{k,p}} \le C\left( \|\div u\|_{W^{k,p}} +\| \o\|_{W^{k,p}} \right).
\ea\ee
\end{lemma}

To estimate $\| \na u\|_{L^{\infty}}$ and $\| \na \n\|_{L^{q}}$ 
we require the following Beale-Kato-Majda type inequality, 
which was established in \cite{K} when $\div u \equiv 0$. 
For further reference, we direct readers to \cite{BKM,HLX1}.
\begin{lemma}\la{bkm}
For $2<q<\infty$, 
there exists a positive constant $C$ 
depending only on $q $ such that for every function $\na u\in W^{1,q}$
\be\ba\la{bkm1}
 \|\na u\|_{L^\infty} \le C \left( \|\div u\|_{L^\infty}+ \|\o\|_{L^\infty} \right)\log \left(e+ \|\na^2 u\|_{L^q} \right)+ C\|\na u\|_{L^2}+C.
\ea\ee
\end{lemma}

Moreover, we state the following Poincar\'e's inequality,
 the proof of which 
 can be found in Danchin-Mucha \cite{DM}.
\begin{lemma}\la{Dl8}
Suppose $\n\in L^2$ and $b \in H^1$. Assume that
\be\ba\la{p7}
\int \n b = 0,\quad  M:=\int \n dx > 0.
\ea\ee
Then, there exists a positive generic constant $C$ such that
\be\ba\la{p8}
\|b\|_{L^2}\leq C\log^\frac{1}{2}\left(e 
 + \frac{\|\n-c\|_{L^2}}{M}\right)
 \|\na b\|_{L^2}, \mathrm{\ for\ all\ } c\in\mr.
\ea\ee
\end{lemma}

Next, the following Zlotnik inequality, 
which plays an important role in obtaining the uniform upper (in time) bound of $\n$, can be found in \cite{ZAA}.
\begin{lemma}\la{zli}
Suppose that the function $y(t)$ is defined on $[0,T]$ and satisfies
\be\ba\nonumber
y'(t)= g(y)+h'(t) \mbox{ on  } [0,T] ,\quad y(0)=y_0, 
\ea\ee
with $ g\in C(\mathbb{R})$ and $y,h\in W^{1,1}(0,T).$ If $g(\infty)=-\infty$
and 
\be\ba\nonumber
h(t_2)-h(t_1) \le N_0 +N_1(t_2-t_1),
\ea\ee
for all $0 \le t_1<t_2\le T$
with some $N_0\ge 0$ and $N_1\ge 0$, then
\be\ba\nonumber
y(t)\le \max\left\{y_0,\overline{\zeta} \right\}+N_0<\infty
\mbox{ on } [0,T],
\ea\ee
where $\overline{\zeta}$ is a constant such
that 
\be\ba\nonumber
g(\zeta)\le -N_1 \quad\mbox{ for }\quad \zeta\ge \overline{\zeta}.
\ea\ee
\end{lemma}

\section{A Priori Estimates \uppercase\expandafter{\romannumeral1}: Upper Bound of $\n$}
In this section, we always assume that $(\n,u)$ is the classical solution of (\ref{ns})--(\ref{i3}) on  $\T^2 \times (0,T]$, and additionally assume that
\be\la{mdsj}\ba
0\leq \n(x,t)\leq \rs:=2\tilde{\n} \quad \mathrm{for\ all\ } (x,t)\in \T^2 \times[0,T],
\ea\ee
where $\tilde{\n}$ is defined in Theorem \ref{th0}.

Moreover, we set
\be\la{dfa1}\ba
A_1^2(t) \triangleq \int \mu \o^2(t)+\frac{G^2(t)}{2\mu+\lambda} dx,
\ea\ee
and
\be\la{dfa2}\ba
A_2^2(t)\triangleq\int \rho(t)|\dot{u}(t)|^2dx.
\ea\ee

We first state the standard energy estimate.

\begin{lemma}\la{l1}
Suppose that $(\n,u)$ is a smooth solution to \eqref{ns}--\eqref{i3}
on  $\mathbb{T}^2 \times (0,T]$, then the following holds:
\be\ba\la{0x1}
\sup_{0\leq t\leq T}\left( \int \frac{1}{2}\rho |u|^2 + \frac{1}{\ga-1} \n^\ga dx \right)
+ \int_0^T  \int \nu (\div u)^2 + \mu \o^2 dxdt 
\leq
E_0,
\ea\ee
where $E_0$ is defined by \eqref{e0}.
\end{lemma}
\begin{proof}
Multiplying $(\ref{ns})_2$ by $u$ and integrating the resulting equation over $\T^2$ by parts,
after using $(\ref{ns})_1$ we derive (\ref{0x1}).
\end{proof}

\begin{lemma}\la{l3}
There exists a positive constant $C$ depending only on $\mu,\ \ga,\ E_0,\ \ol{\n_0}$ 
and $\rs$, such that
\be\la{0xx2}\ba
\int_0^T \int (P-\ol{P})^2 dx dt \le C \nu.
\ea\ee
\end{lemma}
\begin{proof}
First, we set 
\be\la{ppp31}\ba
B(\n,\ol{\n})=\n \int^\n_{\ol{\n}} \frac{P(s)-P(\ol{\n})}{s^2} ds.
\ea\ee
According to (\ref{conv}), there exist two positive constants $M_1$ and $M_2$ 
both depending only on $\ga, \ \ol{\n_0}$ and $\rs$ such that 
\be\la{ppp32}\ba
M_1(\n-\ol{\n})^2 \le M_2 B(\n,\ol{\n})  \le (\n^\ga -\ol{\n}^\ga)( \n-\ol{\n}).
\ea\ee

Next, multiplying $(\ref{ns})_2$ by $-\na (-\Delta)^{-1} (\n -\ol{\n})$ and
integrating the resulting equation over $\T^2$, we can obtain
\be\la{ppp33}\ba
& \int(P-P(\ol{\n}))(\n-\ol{\n})dx \\
&= \left(-\int \n u \cdot \na (-\Delta)^{-1} (\n -\ol{\n}) dx\right)_t 
+ \int \rho u\cdot \na (-\Delta)^{-1} (\n_t) dx 
+ (\mu+\lambda)\int(\n-\ol{\n})\div udx \\
&\quad -\mu \int \p_j u_i \p_j \p_i (-\Delta)^{-1} (\n -\ol{\n}) dx
+ \int \n u_i u_j \p_i \p_j (-\Delta)^{-1} (\n -\ol{\n}) dx \\
& \le \left(-\int \n u \cdot \na (-\Delta)^{-1} (\n -\ol{\n}) dx\right)_t + C\|\n u\|_{L^2}^2 +\nu \| \div u\|_{L^2} \| \n -\ol{\n} \|_{L^2} \\
&\quad +C\|\na u\|_{L^2} \| \n -\ol{\n} \|_{L^2}  +\| \n \|_{L^4} \| u\|^2_{L^4} \| \n -\ol{\n} \|_{L^4} \\
& \le \left(-\int \n u \cdot \na (-\Delta)^{-1} (\n -\ol{\n}) dx\right)_t 
+\varepsilon \| \n -\ol{\n} \|^2_{L^2} + C(\varepsilon) \left( \| \o \|^2_{L^2} + \nu ^2 \| \div u\|^2_{L^2} \right),
\ea\ee
where we have used $(\ref{ns})_1$, (\ref{p8}), (\ref{dc1}), (\ref{gn11}) and H\"older's inequality.

Combining (\ref{ppp32}) and (\ref{ppp33}) and taking $\varepsilon$ suitably small implies
\be\la{ppp34}\ba
M_2 \int B(\n,\ol{\n})dx
&\le \int(\rho^\gamma-\bar{\rho}^\gamma)( \rho - \bar{\rho})dx \\
&\leq -2\left(\int \n u \cdot \na (-\Delta)^{-1} (\n -\ol{\n}) dx \right)_t 
+\widetilde{C} \left(\| \o \|^2_{L^2} + \nu ^2 \| \div u\|^2_{L^2} \right),
\ea\ee
where $\widetilde{C}$ depends only on $\mu,\ \ga,\ \ol{\n_0}$ and $\rs$.

Integrating (\ref{ppp34}) over $(0,T)$ and using (\ref{0x1}) yields
\be\la{ppp35}\ba
\int_0^T \int B(\n,\ol{\n})dxdt \le C \nu,
\ea\ee
where we have used the following estimate
\be\la{ppp36}\ba
\left| \int \n u \cdot \na (-\Delta)^{-1} (\n -\ol{\n}) dx \right|
\le \| \n u\|_{L^2} \| \na (-\Delta)^{-1} (\n -\ol{\n}) \|_{L^2}
\le C \| \n -\ol{\n} \|_{L^2} 
\le C.
\ea\ee
Moreover, based on the definition of $B(\n,\ol{\n})$, we can derive
\be\la{ppp37}\ba
\|P-\ol{P}\|_{L^2}^2 \le \|P-P(\ol{\n})\|_{L^2}^2 \le C \int B(\n,\ol{\n}) dx,
\ea\ee
which together with (\ref{ppp35}) gives (\ref{0xx2}).
\end{proof}

\begin{lemma}\la{l4}
There exists a positive constant $C$ depending only on $\ga,\ \mu,\ \rs,\ \ol{\n_0}$ 
and $E_0$, such that
\be\la{x1000}\ba
\sup_{0\le t\le T} A^2_1 + \int_0^T A^2_2 dt 
\le C \left(1+\nu \| \div u_0 \|^2_{L^2}+ \| \na^{\bot} \cdot u_0 \|^2_{L^2} \right),
\ea\ee
and
\be\la{0x10}\ba
\sup_{0\le t\le T} \si A^2_1 + \int_0^T \si A^2_2 dt \le C,
\ea\ee
with
\be\nonumber
\si(t):=\min\{1,t\}.
\ee
\end{lemma}

\begin{proof}
First, direct calculations show that 	
\be\la{0x11}\ba
\na^{\bot}\cdot \dot u= \frac{D}{Dt}\o +(\p_1u\cdot\na) u^2
-(\p_2u\cdot\na)u^1
 = \frac{D}{Dt}\o + \o \div u , 
\ea\ee
and that
\be\la{0x12}\ba 
\div  \dot u&=\frac{D}{Dt}\div u +(\p_1u\cdot\na) u^1
+(\p_2u\cdot\na)u^2\\&
=\frac{1}{\nu} \frac{D}{Dt}G
+ \frac{1}{\nu} \frac{D}{Dt}(P-\ol{P})
 + 2\nabla u^1\cdot\nabla^{\perp}u^2 + (\div u)^2.	
\ea\ee

Then, we rewrite $(\ref{ns})_2$ as 
\be\la{0x13}\ba
\n\dot{u} = \na G + \mu\na^{\bot}\o.
\ea\ee
Multiplying both sides of $(\ref{ns})_2$ by $2 \dot u$ and then integrating
the resulting equality over $\T^2$ leads to
\be\la{0x14}\ba
& \frac{\rm d}{{\rm d}t} \int \left(\mu \o^2 + \frac{G^2}{\nu}\right)dx + 2\| \sqrt{\n} \dot{u}\|^2_{L^2}\\
& = -\mu \int \o^2\div udx - 
4\int G\nabla u^1
\cdot\nabla^{\perp}u^2dx- 2\int G(\div u)^2dx\\
&\quad +\frac{1}{\nu} \int G^2\div udx +\frac{2\ga}{\nu} \int P G\div udx= \sum_{i=1}^5I_i,
\ea\ee
where we have used (\ref{0x11}) and (\ref{0x12}). Next, we estimate each $I_i$ as follows:

First, combining (\ref{gn11}), (\ref{p2}) and H\"older's inequality leads to
\be\la{0x15}\ba
|I_1| &\le C \| \o\|^2_{L^4} \| \div u\|_{L^2} \\
& \le C \| \o\|_{L^2} \| \na \o\|_{L^2} \| \div u\|_{L^2} \\
& \le C \| \sqrt{\n} \dot{u}\|_{L^2} \| \o\|_{L^2} \| \div u\|_{L^2} \\
& \le \frac{1}{8} \| \sqrt{\n} \dot{u}\|^2_{L^2}+C A^2_1 \| \div u\|^2_{L^2}.
\ea\ee
Next, we will use the idea presented in \cite{HL2} to estimate $I_2$. 
Observing that 
\be\la{0x16}\ba
\na^{\bot} \cdot \na u^1=0,
\quad \div \na^{\bot}u^2=0,
\ea\ee
we can infer from \cite[Theorem II.1]{CLMS} that
\be\la{0x17}\ba
\|\na u^1
\cdot \na^{\bot}u^2\|_{\mathcal{H}^1}\le C\|\na u\|_{L^2}^2.
\ea\ee
Based on the fact that $\mathcal{BMO}$ is the dual space of $\mathcal{H}^1$ (see \cite{FC}), we obtain
\be\la{0x18}\ba
|I_2| &\le C \| G\|_{\mathcal{BMO}} 
\|\na u^1\cdot \na^{\bot}u^2\|_{\mathcal{H}^1} \\
& \le C \| \na G\|_{L^2} \| \na u\|^2_{L^2} \\
& \le C \| \sqrt{\n} \dot{u}\|_{L^2} \| \na u\|^2_{L^2} \\
& \le \frac{1}{8} \| \sqrt{\n} \dot{u}\|^2_{L^2}+ C \| \na u\|^4_{L^2}.
\ea\ee

Then, it follows from (\ref{gn11}), (\ref{p2}), (\ref{gw}) and H\"older's inequality that
\be\la{0x20}\ba
\sum_{i=3}^5I_i &\le \frac{C}{\nu} \int G^2 |\div u|dx+\frac{C}{\nu} \int (P+\ol{P})|G| |\div u|dx \\
&\le \frac{C}{\nu} \| G\|^2_{L^4} \| \div u\|_{L^2}+\frac{C}{\nu} \| G\|_{L^2} \| \div u\|_{L^2} \\
& \le \frac{C}{\nu} \| G\|_{L^2} \| \na G\|_{L^2} \| \div u\|_{L^2}+\frac{C}{\nu} \| \na G\|_{L^2} \| \div u\|_{L^2} \\
& \le \frac{1}{8} \| \sqrt{\n} \dot{u}\|^2_{L^2}+ C A^2_1 \| \na u\|^2_{L^2}
+\frac{C}{\nu^2} \| \div u\|^2_{L^2}.
\ea\ee
Putting (\ref{0x15}), (\ref{0x18}) and (\ref{0x20}) into (\ref{0x14}) implies that
\be\la{0x21}\ba
\frac{d}{dt} A^2_1 + \| \sqrt{\n} \dot{u}\|^2_{L^2} 
& \le C \| \na u\|^4_{L^2} + C A^2_1 \| \na u\|^2_{L^2} + \frac{C}{\nu^2} \| \div u\|^2_{L^2} \\
&\le C A^2_1 \left( A^2_1 + \| \na u\|^2_{L^2} \right)+ \frac{C}{\nu^4} \| P-\ol{P} \|^4_{L^2} + \frac{C}{\nu^2} \| \div u\|^2_{L^2},
\ea\ee
where in the second inequality we have used the following estimate:
\be\la{0x22}\ba
\| \na u\|^2_{L^2} 
& \le C \left( \| \div u \|^2_{L^2} + \| \o \|^2_{L^2} \right) \\
& \le C \left( \frac{1}{\nu^2} \| G \|^2_{L^2} + \frac{1}{\nu^2} \| P-\ol{P} \|^2_{L^2} + A^2_1 \right) \\
& \le C A^2_1 + \frac{C}{\nu^2} \| P-\ol{P} \|^2_{L^2},
\ea\ee
due to (\ref{dc1}).

In addition, by virtue of (\ref{gw}), (\ref{0x1}), (\ref{0xx2}) we get
\be\la{x23}\ba
\int_0^T A^2_1 dt \le C \int_0^T \| \o \|^2_{L^2} + \nu \| \div u \|^2_{L^2}
+\frac{1}{\nu} \| P-\ol{P} \|^2_{L^2} dt \le C,
\ea\ee
which together with (\ref{0x21}), (\ref{0x1}), (\ref{0xx2}) and Gr\"onwall's inequality leads to (\ref{x1000}).

Then, multiplying (\ref{0x21}) by $\si$ and using (\ref{mdsj}) and (\ref{0x1}) gives
\be\la{0x24}\ba
\frac{d}{dt} \left( \si A^2_1 \right) + \si \| \sqrt{\n} \dot{u}\|^2_{L^2} 
& \le \si{'} A^2_1 + C \si A^2_1 \left( A^2_1 + \| \na u\|^2_{L^2} \right)+ \frac{C}{\nu} \| P-\ol{P} \|^2_{L^2} + C \| \div u\|^2_{L^2},
\ea\ee
which combined with (\ref{0x1}), (\ref{x23}) and Gr\"onwall's inequality implies (\ref{0x10}).
\end{proof}

\begin{lemma}\la{g1}
There exists a positive constant $C$ depending only on $\ga,\ \mu,\ E_0,\ \rs,\ \ol{\n_0}$ such that
\be\ba\la{pd11}
\sup_{0\le t\le T}
\si^2 \int\n|\dot u|^2dx
+\int_0^{T} \si^2 \| \na\dot u\|^2_{L^2} dt \le C.
\ea\ee
\end{lemma}
\begin{proof}
First, we rewrite $(\ref{ns})_2$ using (\ref{gw}) as
\be\la{pd12}\ba
\n \dot{u} = \mu \Delta u+\frac{\nu-\mu}{\nu} \na G -\frac{\mu}{\nu} \na(P-\ol{P}).
\ea\ee
Following the approach of \cite{H1,LZ}, we apply the operator $ \dot u^j[\pa/\pa t+\div
(u\cdot)]$ to $ (\ref{pd12})^j.$ Summing with respect to $j,$ and integrating over ${\T^2}$ yields
\be\la{pd13} \ba &
\left(\frac{1}{2}\int\rho|\dot{u}|^2dx \right)_t\\
& = \mu\int\dot{u}^j \left[\Delta u_t^j + \text{div}(u\Delta u^j) \right] dx 
-\frac{\mu}{\nu}\int\dot{u}^j \left[\p_jP_t+\text{div}(u \p_j (P-\ol{P}) \right]dx \\
& \quad + \frac{\nu-\mu}{\nu} \int \dot{u}^j \left[\p_t\p_j G+\text{div}(u\p_j G)\right] dx \\
& \triangleq\sum_{i=1}^{3}N_i. \ea\ee
For $N_1$, integration by parts and applying Young's inequality lead to
\be\la{pd14}\ba
N_1 & =  \mu\int \dot{u}^j \left[\Delta u_t^j
+ \text{div}(u\Delta u^j) \right]dx \\
& = - \mu\int \left[|\nabla\dot{u}|^2 + \p_i\dot{u}^j\p_ku^k\p_iu^j - \p_i\dot{u}^j\p_iu^k\p_ku^j - \p_k\dot{u}^j \p_iu^j\p_iu^k \right]dx \\
&\le -\frac{ 3\mu}{4} \| \na \dot{u} \|_{L^2}^2 + C \| \na u \|_{L^4}^4.
\ea\ee
Similarly, combining $(\ref{ns})_1$ with Young's inequality implies that
\be\la{pd15}\ba
N_2 & = - \frac{\mu}{\nu} \int \dot{u}^j \left[\p_jP_t + \text{div}(u \p_j (P-\ol{P}) \right]dx \\
& = \frac{\mu}{\nu} \int \left[-P^{'}\rho\div \dot u\div u + (P-\ol{P}) \div \dot u\div u 
-(P-\ol{P}) \p_i\dot{u}^j\p_ju^i \right]dx \\
& \le \frac{\mu}{8} \| \na \dot{u} \|_{L^2}^2 + \frac{C}{\nu^2} \| \div u\|^2_{L^2}
+ \frac{C}{\nu^2} \| P-\ol{P} \|^2_{L^4} \| \na u \|^2_{L^4}.
\ea\ee
Next, integrating by parts and using (\ref{p2}), (\ref{0x17}) and Young's inequality,
we derive
\be\la{pd16}\ba
N_3 & = \frac{\nu-\mu}{\nu} \int\dot{u}^j \left[\p_j \p_t G + \text{div}( u \p_j G ) \right] dx \\
& = -\frac{\nu-\mu}{\nu} \int \div \dot{u} \left( \dot{G}-u \cdot \na G \right) dx
+\frac{\nu-\mu}{\nu} \int \dot{u}^j \p_j G \div u + \dot{u}^j u \cdot \na \p_j G dx \\
& = -\frac{\nu-\mu}{\nu} \int \div \dot{u} \dot{G} dx
+\frac{\nu-\mu}{\nu} \int -G \div \dot{u} \div u - G \dot{u}^j \p_j \div u
- \dot{u}^j \p_j u \cdot \na G dx \\
& = -\frac{\nu-\mu}{\nu} \int \div \dot{u} \dot{G} dx
+\frac{\nu-\mu}{\nu} \int -G \div \dot{u} \div u + G \p_j u \cdot \na \dot{u}^j dx \\
& = -\frac{\nu-\mu}{\nu} \int \div \dot{u} \dot{G} dx
+\frac{\nu-\mu}{\nu} \int G \left( \na u_1 \cdot \na^\bot \dot{u}^2 - \na u_2 \cdot \na^\bot \dot{u}^1 \right) dx \\
& \le -\frac{\nu-\mu}{\nu} \int \div \dot{u} \dot{G} dx + C \| \na G \|_{L^2} \| \na u \|_{L^2} \| \na \dot{u} \|_{L^2} \\
& \le -\frac{\nu-\mu}{\nu} \int \div \dot{u} \dot{G} dx + 
\frac{\mu}{8} \| \na \dot{u} \|^2_{L^2}
+ C \| \sqrt{\n} \dot{u} \|^2_{L^2} \| \na u \|^2_{L^2},
\ea\ee
where we have used the following fact:
\be\la{pd17}\ba
\sum_{j=1}^{2} \p_j u \cdot \na \dot{u}^j
= \div u \div \dot{u} + \na u^1 \cdot \na^\bot \dot{u}^2 - \na u^2 \cdot \na^\bot \dot{u}^1.
\ea\ee
We now estimate the first term in the last line of (\ref{pd16}). Using the definition of $G$ and $(\ref{ns})_1$, we obtain
\be\la{pd18}\ba
\div \dot{u} & = (\div u)_t + \p_i u^j \p_j u^i + u \cdot \na \div u \\
& = \frac{1}{\nu} \dot{G} + \frac{1}{\nu} (P_t+u \cdot \na P) - \frac{1}{\nu} \ol{P_t}
+ \p_i u^j \p_j u^i \\
& = \frac{1}{\nu} \dot{G} - \frac{1}{\nu} \ga P \div u + \frac{\ga-1}{\nu} \int P \div u dx + \p_i u^j \p_j u^i,
\ea\ee
which together with Young's inequality yields
\be\la{pd19}\ba
-\frac{\nu-\mu}{\nu} \int \div \dot{u} \dot{G} dx
& = -\frac{\nu-\mu}{\nu^2} \| \dot{G} \|^2_{L^2} 
-\frac{\nu-\mu}{\nu} \int \dot{G} \p_i u^j \p_j u^i dx \\
& \quad + \frac{\nu-\mu}{\nu^2} \ga \int \dot{G} P \div u dx
-\frac{\nu-\mu}{\nu^2} (\ga-1) \int P \div u dx \int \dot{G} dx \\
& \le -\frac{\nu-\mu}{2\nu^2} \| \dot{G} \|^2_{L^2}+\frac{C}{\nu} \| \div u \|^2_{L^2}
-\frac{\nu-\mu}{\nu} \int \dot{G} \p_i u^j \p_j u^i dx.
\ea\ee
For the last term in the final line of (\ref{pd19}), integration by parts results in
\be\la{pd110}\ba
\int \dot{G} \p_i u^j \p_j u^i dx 
& = \int ( G_t + u \cdot \na G ) \p_i u \cdot \na u^i dx \\
& = \frac{d}{dt} \left( \int G \p_i u \cdot \na u^i dx \right) 
- 2 \int G \p_i u \cdot \na u^i_t dx \\
& \quad -\int G \div u \p_i u \cdot \na u^i dx 
-2 \int G u \cdot \na \p_i u \cdot \na u^i dx \\
& = \frac{d}{dt} \left( \int G \p_i u \cdot \na u^i dx \right) 
- 2 \int G \p_i u \cdot \na \dot{u}^i dx \\
& \quad +2 \int G \p_i u \cdot \na u \cdot \na u^i dx
-\int G \div u \p_i u \cdot \na u^i dx,
\ea\ee
which implies that
\be\la{pd111}\ba
-\frac{\nu-\mu}{\nu} \int \dot{G} \p_i u^j \p_j u^i dx
& =-\frac{\nu-\mu}{\nu} \frac{d}{dt} \left( \int G \p_i u \cdot \na u^i dx \right)
+ \frac{2(\nu-\mu)}{\nu} \int G \p_i u \cdot \na \dot{u}^i dx \\
& \quad - \frac{2(\nu-\mu)}{\nu} \int G \p_i u \cdot \na u \cdot \na u^i dx
+\frac{\nu-\mu}{\nu} \int G \div u \p_i u \cdot \na u^i dx.
\ea\ee
Next, we estimate the last three terms on the right-hand side of (\ref{pd111}) separately.
It follows from (\ref{pd17}), (\ref{gw}), (\ref{p2}) and Young's inequality that
\be\la{pd112}\ba
\frac{2(\nu-\mu)}{\nu} \int G \p_i u \cdot \na \dot{u}^i dx
& \le 2 \left| \int G \left( \div u \div \dot{u} + \na u^1 \cdot \na^\bot \dot{u}^2 - \na u^2 \cdot \na^\bot \dot{u}^1 \right) dx \right| \\
& \le C \| \na \dot{u} \|_{L^2} \| G \|_{L^4} \| \div u \|_{L^4}
+ C \| \na G \|_{L^2} \| \na \dot{u} \|_{L^2} \| \na u \|_{L^2} \\
& \le \frac{\mu}{8} \| \na \dot{u} \|^2_{L^2} + \frac{C}{\nu^2} \| G \|^4_{L^4}
+ \frac{C}{\nu^2} \| P-\ol{P} \|^4_{L^4} 
+ C \| \sqrt{\n} \dot{u} \|^2_{L^2} \| \na u \|^2_{L^2}.
\ea\ee
In addition, by virtue of (\ref{gw}) and Young's inequality, it holds that
\be\la{pd113}\ba
- \frac{2(\nu-\mu)}{\nu} \int G \p_i u \cdot \na u \cdot \na u^i dx
& \le 2 \left| \int G \left( (\div u)^3+3\div u \na u^1 \cdot \na^\bot u^2 \right) dx \right| \\
& \le C \| G \|_{L^4} \| \div u \|_{L^4} \| \na u \|^2_{L^4} \\
& \le \frac{C}{\nu^2} \| G \|^4_{L^4} + \frac{C}{\nu^2} \| P-\ol{P} \|^4_{L^4} 
+ C \| \na u \|^4_{L^4},
\ea\ee
where in the first inequality, we have used the following fact:
\be\la{pd114}\ba
\sum_{i=1}^{2} \p_i u \cdot \na u \cdot \na u^i
= (\div u)^3 +3 \div u \na u^1\cdot \na^\bot \dot{u}^2.
\ea\ee
Similar to (\ref{pd113}) we have
\be\la{pd115}\ba
\int G \div u \p_i u \cdot \na u^i dx 
\le \frac{C}{\nu^2} \| G \|^4_{L^4} + \frac{C}{\nu^2} \| P-\ol{P} \|^4_{L^4} 
+ C \| \na u \|^4_{L^4}.
\ea\ee
Substituting (\ref{pd112}), (\ref{pd113}) and (\ref{pd115}) into (\ref{pd111}), we derive
\be\la{pd116}\ba
-\frac{\nu-\mu}{\nu} \int \dot{G} \p_i u^j \p_j u^i dx
& \le -\frac{\nu-\mu}{\nu} \frac{d}{dt} \left( \int G \p_i u \cdot \na u^i dx \right)
+ \frac{\mu}{8} \| \na \dot{u} \|^2_{L^2} + \frac{C}{\nu^2} \| G \|^4_{L^4} \\
& \quad + \frac{C}{\nu^2} \| P-\ol{P} \|^4_{L^4} 
+ C \| \sqrt{\n} \dot{u} \|^2_{L^2} \| \na u \|^2_{L^2} + C \| \na u \|^4_{L^4}.
\ea\ee
Combining (\ref{pd16}), (\ref{pd19}) and (\ref{pd116}) yields
\be\la{pd117}\ba
N_3 & \le -\frac{\nu-\mu}{\nu} \frac{d}{dt} \left( \int G \p_i u \cdot \na u^i dx \right)
-\frac{\nu-\mu}{2\nu^2} \| \dot{G} \|^2_{L^2}+\frac{C}{\nu} \| \div u \|^2_{L^2}
+ \frac{\mu}{4} \| \na \dot{u} \|_{L^2} \\
& \quad + C \| \sqrt{\n} \dot{u} \|^2_{L^2} \| \na u \|^2_{L^2} + \frac{C}{\nu^2} \| G \|^4_{L^4}
+ \frac{C}{\nu^2} \| P-\ol{P} \|^4_{L^4} + C \| \na u \|^4_{L^4}.
\ea\ee
Putting (\ref{pd14}), (\ref{pd15}) and (\ref{pd117}) into (\ref{pd13}), we obtain
\be\la{pd118}\ba
&\left(\int\rho|\dot{u}|^2dx + \frac{2(\nu-\mu)}{\nu} \int G \p_i u \cdot \na u^i dx \right)_t
+ \frac{3\mu}{4} \| \na \dot{u} \|_{L^2}^2 + \frac{\nu-\mu}{\nu^2} \| \dot{G} \|^2_{L^2} \\
& \le \frac{C}{\nu} \| \div u \|^2_{L^2} + C \| \sqrt{\n} \dot{u} \|^2_{L^2} \| \na u \|^2_{L^2} + \frac{C}{\nu^2} \| G \|^4_{L^4}
+ \frac{C}{\nu^2} \| P-\ol{P} \|^4_{L^4} + C \| \na u \|^4_{L^4}.
\ea\ee
Multiplying (\ref{pd118}) by $\si^2$ and applying (\ref{0x10}) and (\ref{gw}) implies that
\be\la{pd119}\ba
&\left( \si^2 \int\rho|\dot{u}|^2dx + \frac{2(\nu-\mu)}{\nu} \si^2 \int G \p_i u \cdot \na u^i dx \right)_t
+ \frac{3\mu}{4} \si^2 \| \na \dot{u} \|_{L^2}^2 + \frac{\nu-\mu}{\nu^2} \si^2 \| \dot{G} \|^2_{L^2} \\
& \le 2 \si^\prime \si \left(\int\rho|\dot{u}|^2dx + \frac{2(\nu-\mu)}{\nu} \int G \p_i u \cdot \na u^i dx \right)
+\frac{C}{\nu} \| \div u \|^2_{L^2} + C \si \| \sqrt{\n} \dot{u} \|^2_{L^2} \\
& \quad + \frac{C}{\nu^2} \si^2 \| G \|^4_{L^4}
+ \frac{C}{\nu^2} \si^2 \| P-\ol{P} \|^4_{L^4} + C \si^2 \| \na u \|^4_{L^4}.
\ea\ee
Moreover, it follows from (\ref{gw}), (\ref{p2}) and Young's inequality that
\be\la{pd120}\ba
\frac{2(\nu-\mu)}{\nu} \int G \p_i u \cdot \na u^i dx 
& \le 2 \left| \int G (\div u)^2 + 2G\na u^1 \cdot \na^\bot u^2 dx \right| \\
& \le \frac{C}{\nu^2} \| G \|^3_{L^3} + \frac{C}{\nu^2} \| P-\ol{P} \|^3_{L^3}
+C \| \na G \|_{L^2} \| \na u \|^2_{L^2} \\
& \le \frac{C}{\nu^2} \| G \|^2_{L^2} \| \na G \|_{L^2} + \frac{C}{\nu^2} \| P-\ol{P} \|^3_{L^3}
+C \| \sqrt{\n} \dot{u} \|_{L^2} \| \na u \|^2_{L^2} \\
& \le \frac14 \| \sqrt{\n} \dot{u} \|^2_{L^2}
+\frac{C}{\nu^4} \| G \|^4_{L^2} + \frac{C}{\nu^2} \| P-\ol{P} \|^3_{L^3}
+C \| \na u \|^4_{L^2}.
\ea\ee
By virtue of (\ref{gn11}), (\ref{p2}) and (\ref{0x10}), we have
\be\la{pd121}\ba
\nu^{-2} \int_0^{T} \si^2 \| G \|^4_{L^4}dt 
& \le C \nu ^{-2} \int_0^{T} \si^2 \| G \|^2_{L^2}\| \na G \|^2_{L^2} dt
\le C.
\ea\ee
Combining this with (\ref{dc1}), we derive
\be\la{pd122}\ba
\int_0^{T} \si^2 \|\na u\|^4_{L^4}dt 
& \le C \int_0^{T} \si^2 \|\div u\|^4_{L^4} + \si^2 \|\o \|^4_{L^4} dt \\
& \le C \int_0^{T} \si^2 \nu^{-4} \left( \| G \|^4_{L^4} + \| P-\ol{P} \|^4_{L^4} \right) 
+ \si^2 \|\o \|^2_{L^2} \| \na \o \|^2_{L^2}dt \\
& \le C.
\ea\ee

Integrating (\ref{pd119}) over $(0,T)$ and applying (\ref{0x10}), 
(\ref{pd120}), (\ref{pd121}), (\ref{pd122}) and H\"older's inequality, we obtain (\ref{pd11}) and complete the proof of Lemma \ref{g1}.
\end{proof}

Next, we employ an approach referring to the work \cite{CL} to derive the exponential decay estimate.
\begin{lemma}\la{g3}
For any $s\in [1,\infty)$,
there exist positive constants $C,\ K_0,\ \nu_0$, where $C$
depends only on $s,\ \ga,\ \mu,\ E_0,\ \ol{\n_0},\ \rs$;
$K_0$ depends only on $\mu,\ \ga, \ \ol{\n_0},\ \rs$;
and $\nu_0$ depends on $\ga,\ \mu,\ \ol{\n_0},\ \rs$,
such that when $\nu \geq \nu_0$, it holds that
\be\la{pd301}\ba
\|P-\ol{P}\|_{L^2}^2 + \| \n-\ol{\n}\|^{s}_{L^s} \le C e^{-2 \alpha_0 t},
\ea\ee
\be\la{pd302}\ba
\sup_{1 \le t \le T}\left(e^{ \alpha_0 t}\left( \| \o\|^2_{L^2} +\nu \| \div u\|^2_{L^2} \right) \right) +
\int_1^T e^{ \alpha_0 t} \|\sqrt{\n}\dot{u}\|^2_{L^2}dt \le C,
\ea\ee 
\be\ba\la{pd303}
\sup_{1 \le t\le T} \left(e^{ \alpha_0 t} \| \sqrt{\n} \dot u \|^2_{L^2} \right)
+\int_1^{T} e^{ \alpha_0 t} \| \na\dot u\|^2_{L^2} dt \le C,
\ea\ee
where we denote $\alpha_0=\frac{K_0}{\nu}$.
\end{lemma}
\begin{proof}
First, according to the proof of Lemma \ref{l3}, there exist two positive constants $M_2$ and $\widetilde{C}$,
where $M_2$ depends only on $\ga,\ \ol{\n_0},\ \rs$;
 $\widetilde{C}$ depends only on $\mu,\ \ga,\ \ol{\n_0},\ \rs$, 
 such that
\be\la{pd36}\ba
M_2 \int B(\n,\ol{\n})dx
& \leq -2\left(\int \n u \cdot \na (-\Delta)^{-1} (\n -\ol{\n}) dx \right)_t 
+\widetilde{C} \left( \| \o \|^2_{L^2} + \nu ^2 \| \div u\|^2_{L^2} \right),
\ea\ee
where $B(\n,\ol{\n})$ is defined in (\ref{ppp31}).

On the other hand, the standard energy estimate implies that
\be\la{pd37}\ba
\frac{d}{dt} \left(\frac{1}{2}\|\sqrt{\n} u\|^2_{L^2}+\int B(\n,\ol{\n})dx  \right) + \mu \| \o\|^2_{L^2} +\nu \| \div u\|^2_{L^2}=0.
\ea\ee
Then, multiplying (\ref{pd36}) by $\frac{1}{2 \widetilde{C} \nu}$, 
and adding the resulting equation to (\ref{pd37}), we derive
\be\la{pd38}\ba
Z'(t) &\le -\frac{M_2}{2 \widetilde{C} \nu} \int B(\n,\ol{\n})dx +\left(\frac{1}{2 \nu} -\mu\right)  
\| \o\|^2_{L^2} -\frac{\nu}{2} \| \div u\|^2_{L^2},
\ea\ee
where
\be\la{pd39}\ba
Z(t): =\frac{1}{2}\|\sqrt{\n} u\|^2_{L^2}+\int B(\n,\ol{\n})dx +\frac{1}{\widetilde{C} \nu} \int \n u \cdot \na (-\Delta)^{-1} (\n -\ol{\n}) dx.
\ea\ee
Next, it follows from (\ref{ppp32}), Sobolev embedding and Poincar\'e's inequality that
\be\la{pd310}\ba
\left| \int \n u \cdot \na (-\Delta)^{-1} (\n -\ol{\n}) dx \right|
&\le \| \n u\|_{L^2} \| \na (-\Delta)^{-1} (\n -\ol{\n}) \|_{L^2} \\
&\le C \| \sqrt{\n} u\|^2_{L^2} + C \| \n -\ol{\n} \|^2_{L^2} \\
&\le M_3 \left(\frac{1}{2} \| \sqrt{\n} u\|^2_{L^2} + \int B(\n,\ol{\n})dx \right),
\ea\ee
where $M_3$ depends only on $\rs,\ \ga$ and $\ol{\n_0}$.

Therefore, combining (\ref{pd39}) and (\ref{pd310}) yields
\be\la{pd311}\ba
\frac{1}{2} \left(\frac{1}{2} \| \sqrt{\n} u\|^2_{L^2} + \int B(\n,\ol{\n})dx \right) \le Z(t) \le 2\left(\frac{1}{2} \| \sqrt{\n} u\|^2_{L^2} + \int B (\n,\ol{\n})dx \right),
\ea\ee
provided $\nu \ge \frac{2M_3}{\widetilde{C}}$. 

Moreover, we can conclude from (\ref{dc1}) and (\ref{p8}) that
\be\la{pd312}\ba 
\int\n |u|^2dx\le C\|\na u\|_{L^2}^2 \le M_4 \left( \| \o\|^2_{L^2}+\| \div u\|^2_{L^2} \right),
\ea\ee
where $M_4$ depends only on $\rs$ and $\ol{\n_0}$.

Combining (\ref{pd312}), (\ref{pd38}) and (\ref{pd311}),
we can infer that for any $\nu \ge
\max \left\{\frac{2M_3}{\widetilde{C}},\frac{1}{\mu}\right\}$,
\be\la{pd313}\ba
Z'(t) &\le-\frac{M_2}{2 \widetilde{C} \nu} \int B(\n,\ol{\n})dx-\frac{\mu}{2 M_4} \int\n |u|^2dx  \\
&\le -4 \alpha_0 \left(\frac{1}{2} \| \sqrt{\n} u\|^2_{L^2} + \int B(\n,\ol{\n})dx \right) \\
&\le -2 \alpha_0 Z(t),
\ea\ee
with $\alpha_0 = \min\left\{\frac{M_2}{8 \widetilde{C} \nu}, \frac{\mu}{4 M_4} \right\} $.

We select a suitably large $\nu_0$ such that
\begin{equation}\la{pd314}
\nu_0 = \max \left\{
\frac{2M_3}{\widetilde{C}},
\frac{1}{\mu},
\frac{M_2 M_4}{2\mu \widetilde{C}}
\right\}.
\end{equation}
When $\nu \ge \nu_0,\  \alpha_0=\frac{K_0}{\nu}$ with $K_0=\frac{M_2}{8 \widetilde{C}}$, and $K_0$ depends only on $\mu,\ \ga, \ \ol{\n_0}$ and $\rs$,
we deduce from (\ref{pd311}), (\ref{pd313}) and Gr\"onwall's inequality that for any $t \ge 0$
\be\la{pd315}\ba
\int \left(\frac{1}{2} \n |u|^2 + B(\n,\ol{\n}) \right) dx \le C e^{-2 \alpha_0 t},
\ea\ee
which together with (\ref{pd37}) leads to
\be\la{pd316}\ba
\int_0^T \left(\mu \| \o\|^2_{L^2} +\nu \| \div u\|^2_{L^2} \right) e^{\alpha_0 t} dt \le C,
\ea\ee
where the constant $C$ depends only on $\ga,\ \mu,\ E_0,\ \ol{\n_0},\ \rs$.

Additionally, combining (\ref{ppp37}) and (\ref{pd315}) implies
\be\la{pd318}\ba
\|P-\ol{P}\|_{L^2}^2 + \| \n-\ol{\n}\|^{2\ga}_{L^{2\ga}} \le C e^{-2 \alpha_0 t},
\ea\ee
which together with H\"older's inequality gives (\ref{pd301}).

Subsequently, multiplying (\ref{0x21}) by $ e^{\alpha_0 t}$ and integrating the resulting equation over $\T^2 \times [1,T]$,
we can infer from (\ref{0x10}), (\ref{pd316}) and (\ref{pd318}) that for any $T \ge 1$,
\be\la{pd319}\ba
& \sup_{1 \le t \le T} \left( e^{\alpha_0 t} A^2_1 \right) 
+ \int_1^T e^{\alpha_0 t} \| \sqrt{\n} \dot{u}\|^2_{L^2} dt \\
&\le C + \frac{C}{\nu} \int_1^T e^{\alpha_0 t} A^2_1 dt + C \int_1^T e^{\alpha_0 t} A^2_1 \left( A^2_1 + \| \na u\|^2_{L^2} \right) dt \\
& \quad + \frac{C}{\nu^4} \int_1^T e^{-\alpha_0 t} dt
+C\int_1^T e^{\alpha_0 t} \| \div u\|^2_{L^2} dt \\
& \le C,
\ea\ee
where in the second inequality we have used the following estimate:
\be\la{pd3190}\ba
\int_1^T e^{\alpha_0 t} A^2_1 dt
& \le C \int_1^T e^{\alpha_0 t} \left( \| \o \|^2_{L^2} + \nu \| \div u \|^2_{L^2}
+\frac{1}{\nu} \| P-\ol{P} \|^2_{L^2} \right)  dt \\
& \le C+\frac{C}{\nu} \int_1^T e^{-\alpha_0 t} dt \\
& \le C,
\ea\ee
due to (\ref{gn11}), (\ref{gw}), (\ref{pd315}), (\ref{pd316}) and (\ref{pd318}).

Finally, we multiply (\ref{pd118}) by $e^{\alpha_0 t}$ and integrate the resulting equation over $[1,T]$.
Then by making use of (\ref{pd11}), (\ref{pd316}), (\ref{pd318}) and (\ref{pd319}), we are able to derive
\be\la{pd320}\ba
& \sup_{1 \le t\le T} \left(e^{\alpha_0 t} \| \sqrt{\n} \dot{u}\|^2_{L^2} \right) 
+\int_1^T e^{\alpha_0 t} \| \na \dot{u}\|^2_{L^2} dt \\
& \le C + C \int_1^T e^{\alpha_0 t} \| \sqrt{\n} \dot{u}\|^2_{L^2} dt
+ C \nu^{-2} \int_1^T e^{\alpha_0 t} \| G \|^4_{L^4} dt \\ 
& \quad + C \nu^{-2} \int_1^T e^{\alpha_0 t} \| P-\ol{P} \|^4_{L^4} dt
+ C \int_1^T e^{\alpha_0 t} \|\nabla u\|^4_{L^4}dt \\
& \le C,
\ea\ee
where in the last inequality we have used H\"older's inequality and the following estimates:
\be\la{pd321}\ba
\nu^{-2} \int_1^T e^{\alpha_0 t} \| G \|^4_{L^4} dt 
& \le C \nu^{-2} \int_1^{T} e^{\alpha_0 t} \| G \|^2_{L^2} \| \na G \|^2_{L^2} dt \\
& \le C \nu^{-1} \int_1^{T} e^{\alpha_0 t} \| \sqrt{\n} \dot{u} \|^2_{L^2} dt \\
& \le C,
\ea\ee
and
\be\la{pd322}\ba
\int_1^T e^{\alpha_0 t} \|\na u\|^4_{L^4} dt 
& \le C \int_1^{T} e^{\alpha_0 t} \left(  \|\div u\|^4_{L^4} + \|\o \|^4_{L^4} \right) dt \\
& \le C \int_1^{T} e^{\alpha_0 t} \left( \nu^{-4} \|G\|^4_{L^4} 
+ \nu^{-4} \|P-\pb\|^4_{L^4} + \|\o \|^2_{L^2} \| \na \o \|^2_{L^2} \right) dt \\
& \le C + C\int_1^{T} e^{\alpha_0 t} \| \sqrt{\n} \dot{u} \|^2_{L^2} dt \\
& \le C,
\ea\ee
due to (\ref{gn11}), (\ref{p2}), (\ref{dc1}), (\ref{0x10}) and (\ref{pd318}).
\end{proof}

\begin{lemma}\la{g5}
There exists a positive constant 
$\nu_1$ depending on  $\ga,\ \mu,\ \|\n_0\|_{L^1\cap L^\infty},
\ \| \na u_0\|_{L^2}$ and $E_0$,
such that, if $(\n,u)$ satisfies that
\be\ba\la{pd51}
\sup_{0\leq t\leq T} \|\n\|_{L^\infty} 
\leq 2 \tilde{\n},
\ea\ee
then 
\be\ba\la{pd52}
\sup_{0\leq t\leq T} \|\n\|_{L^\infty} 
\leq \frac{3}{2} \tilde{\n},
\ea\ee
provided $\nu \geq \nu_1$.
\end{lemma}

\begin{proof}
First, with the defination of $G$ in (\ref{gw}), we can rewrite $(\ref{ns})_1$ as
\be\la{pd53}\ba
D_t \n =g(\n)+h'(t),
\ea\ee
where 
\be\la{pd54}\ba
D_t\n\triangleq\n_t+u \cdot\nabla \n ,\quad
g(\n)\triangleq-\frac{\n}{\nu} ( \n^\ga-\ol{P} ),
\quad h(t)\triangleq-\frac{1}{\nu} \int_0^t \n G ds. 
\ea\ee
For any $0 \le t_1 <t_2 \le \si(T) $, it follows from (\ref{0x1}), (\ref{pd51}) and H\"older's inequality that
\be\la{pd55}\ba 
|h(t_2)-h(t_1)| \le \frac{C}{\nu} \int_0^{\si(T)} \| G\|_{L^\infty} dt.
\ea\ee
Then, we conclude from (\ref{x1000}), (\ref{pd11}), (\ref{gn12}), (\ref{p2}) and H\"older's inequality that
\be\la{pd5500}\ba 
\int_0^{\si(T)} \| G\|_{L^\infty} dt 
& \le C \int_0^{\si(T)} \| G\|_{L^2}^{3/8}
\| \na G\|_{L^5}^{5/8} dt \\
& = C \int_0^{\si(T)} \| G \|_{L^2} ^{3/8} \left(\si^2\|\na G \|^{2}_{L^5} \right)^{5/16} 
\si^{ -5/8 } dt  \\
& \le C \nu ^{3/8} \int_0^{\si(T)} \left(\si^2\|\n^{1/2}\dot u\|^{2}_{L^2} +\si^2\|\na\dot u\|^{2}_{L^2}\right)^{5/16}
\si^{-5/8} dt  \\
& \le C \nu^{3/8} \left(\int_0^{\si(T)} \si^{-10/11}dt\right)^{11/16} \\
& \le C \nu^{3/8},
\ea\ee
which together with (\ref{pd55}) yields
\be\la{pd5511}\ba 
|h(t_2)-h(t_1)| \le K_1 \nu^{-5/8},
\ea\ee
where $K_1$ is a constant depending only on $\mu,\ \ga,\ \|\n_0\|_{L^1\cap L^\infty},\ \| \na u_0\|_{L^2}$ and $E_0$.

We take $N_0$, $N_1$ and $\overline{\zeta}$ in Lemma \ref{zli} as follows:
\be\la{pd5522}\ba
N_0=K_1 \nu^{-5/8}, \quad N_1=0, \quad \overline{\zeta}=(\ol{P})^{1/\ga},
\ea\ee
which implies
\be\la{pd056}\ba
g(\zeta)=-\frac{\zeta}{\nu} ( \zeta^\ga-\ol{P} ) \le -N_1=0 \quad  \mathrm{for\ all\ }
\zeta \ge \overline{\zeta}=(\ol{P})^{1/\ga}.
\ea\ee
Lemma \ref{zli} thus results in
\be\la{pd56}\ba
\sup_{t\in [0,\si(T)]} \|\n\|_{L^\infty} 
& \le \| \n_0\|_{L^\infty} + \left( (\ga-1)E_0 \right)^{1/\ga}
+K_1 \nu^{-5/8}
\le \frac{5}{4} \tilde{\n},
\ea\ee
provided $\nu \geq  \max\left\{ \nu_0, 
\left( \frac{4K_1}{ \tilde{\n} } \right)^{8/5}\right\}$,
where $\nu_0$ is given by (\ref{pd314}).

On the other hand, for $t\in [\si(T),T]$, based on (\ref{pd301}), (\ref{pd302}), (\ref{pd303}), (\ref{gn12}) and (\ref{p2}), we derive that
\be\la{pd57}\ba 
\int_{\si(T)}^{T} \| G\|_{L^\infty} dt
& \le C \int_{\si(T)}^{T} \| G\|_{L^2}^{1/3} \| \na G\|^{2/3}_{L^{4}} dt \\
& \le C \nu^{1/6} \int_{\si(T)}^{T} e^{-\frac{1}{6} \alpha_0 t}
\left( \| \sqrt{\n} \dot u\|_{L^2} + \|\na\dot u\|_{L^2}\right)^{2/3} dt \\
& \le C \nu^{1/6} \int_{\si(T)}^{T} e^{-\frac{1}{6} \alpha_0 t} \| \sqrt{\n} \dot u\|_{L^2}^{2/3}
+ e^{-\frac{1}{2} \alpha_0 t}
\left( e^{\alpha_0 t} \| \na \dot u\|^{2}_{L^2} \right)^{1/3} dt \\
& \le C \nu^{1/6} \left( \int_{\si(T)}^{T} e^{-\frac{1}{4} \alpha_0 t} dt \right)^{2/3}
\left( \int_{\si(T)}^{T} \| \sqrt{\n} \dot u\|^2_{L^2} dt \right)^{1/3} \\
& \quad + C \nu^{1/6} \left( \int_{\si(T)}^{T} e^{-\frac{3}{4} \alpha_0 t} dt \right)^{2/3}
\left( \int_{\si(T)}^{T} e^{\alpha_0 t} \| \na \dot u\|^{2}_{L^2} dt \right)^{1/3} \\
& \le C \nu^{5/6}.
\ea\ee

Therefore, (\ref{pd57}), (\ref{pd51}), (\ref{pd54}) and H\"older's inequality show that for all $\si(T)\le t_1 < t_2\le T$, 
\be\la{pd58}\ba  
|h(t_2)-h(t_1)| 
\le \frac{C}{\nu} \int_{t_1}^{t_2} \| G(\cdot,t)\|_{L^\infty}dt
\le K_2 \nu^{-1/6},
\ea\ee
where $K_2$ is a constant depending only on $\mu,\ \ga,\ \|\n_0\|_{L^1\cap L^\infty}$ and $E_0$.
We choose $N_1$, $N_0$ and $\overline{\zeta}$ in Lemma \ref{zli} as follows:
\be\la{pd59}\ba
N_0=K_2 \nu^{-1/6},\quad  N_1=0,\quad 
\overline{\zeta}=(\ol{P})^{1/\ga}.
\ea\ee
Moreover, we set
\be\ba\la{pdd}
\nu_1 := \max \left\{ \nu_0,
\left(\frac{4K_1}{\tilde{\n}}\right)^{8/5},
\left(\frac{4K_2}{\tilde{\n}}\right)^{6} \right\},
\ea\ee
which together with (\ref{pd056}), (\ref{pd56}) and Lemma \ref{zli}
leads to when $\nu\geq \nu_1$
\be\la{pd511}\ba
\sup_{t\in	[\si(T),T]} \|\n\|_{L^\infty} \le \frac{5}{4} \tilde{\n}
+K_2 \nu^{-1/6}  \le \frac{3}{2} \tilde{\n}.
\ea\ee
Combining this with (\ref{pd56}) yields (\ref{pd52}) and completes the proof of Lemma \ref{g5}. 
\end{proof}

\section{A Priori Estimates \uppercase\expandafter{\romannumeral2}: Higher Order Estimates}
In this section, we always assume $\nu \ge \nu_1$ is a fixed constant, where $\nu_1$ is determined in Lemma \ref{g5}, 
and let $(\n,u)$ be a smooth solution of (\ref{ns})--(\ref{i3}) on $\T^2 \times (0,T]$ satisfying (\ref{pd51}). 
Next, we will derive the higher-order estimates for the smooth solution $(\n,u)$, which are similar to the arguments in \cite{HLX2,LLL}. 

First, we assume that the initial data $(\n_0,u_0)$ satisfy (\ref{ssol1}).
\begin{lemma}\la{s21}
There exists a positive constant $C$ depending only on  
$T$, $\ga$, $\mu$, $\nu$, $\| \rho_0 \|_{L^1 \cap L^\infty }$, $E_0$ and
$\| u_0 \|_{H^1}$, such that
\be\la{s411} \ba
\sup_{0\le t\le T} \si \int\n|\dot u|^2dx+\int_0^{ T} \si \|\na\dot u\|^2_{L^2}dt\le C.
\ea\ee	

\end{lemma}
\begin{proof}
First, it follows from (\ref{pt1}), (\ref{0x1}), (\ref{x1000}) and (\ref{pd51}) that
\be\la{s418} \ba
\sup_{0\le t\le T} \left( \| \n \|_{L^\infty} + \| u\|_{H^1} \right) + \int_0^T \left( \| \na u\|^2_{L^2}+ \| \n^{1/2} \dot{u} \|^2_{L^2} \right)dt \le C.
\ea \ee
On the other hand, we deduce from (\ref{pd118}) and (\ref{s418}) that
\be\la{s417}\ba
&\left(\int\rho|\dot{u}|^2dx + \frac{2(\nu-\mu)}{\nu} \int G \p_i u \cdot \na u^i dx \right)_t
+ \frac{3\mu}{4} \| \na \dot{u} \|_{L^2}^2 \\
& \le C \| \na u \|^2_{L^2} + C \| \sqrt{\n} \dot{u} \|^2_{L^2}
+ \frac{C}{\nu^2} \| P-\ol{P} \|^4_{L^4} + C \| \na u \|^4_{L^4}.
\ea\ee
Moreover, similar to (\ref{pd18}) and (\ref{pd19}), we derive
\be\la{s419} \ba
\int_0^{T}\|\na u\|^4_{L^4}dt  \le C.
\ea\ee
Multiplying (\ref{s417}) by $\si$ and integrating the resulting equation over $(0,T)$,
we obtain (\ref{s411}) after using (\ref{pd120}), (\ref{s418}) and (\ref{s419}).
\end{proof}
\begin{lemma}\la{s22}
For any $2 < p < \infty$, there exists a positive constant $C$ depending only on 
$T,\ p,\ \ga,\ \mu,\ \nu,\ E_0,\ \| \rho_0 \|_{L^1 \cap L^\infty }$ and 
$\| u_0 \|_{H^1}$, such that
\be\la{s421} \ba
&\int_0^T \left(\|G\|_{L^\infty}+\|\na G\|_{L^p}+\|\o\|_{L^\infty}+\|\na \o\|_{L^p}+\| \rho \dot u\|_{L^p}\right)^{1+1 /p}dt \\& +\int_0^Tt\left(\|\na G\|_{L^p}^2+\|\na \o\|_{L^p}^2+\|   \dot u\|_{H^1}^2\right) dt \le C. \ea\ee
\end{lemma}
\begin{proof}
First, by using H\"{o}lder's inequality and (\ref{pt1}) we obtain
\be\la{s422} \ba
\| \rho \dot u\|_{L^p} & \le
C\| \rho \dot u\|_{L^2}^{2(p-1)/(p^2-2)}\|\dot u\|_{L^{p^2}}^{p(p-2)/(p^2-2)}\\ & \le
C\| \rho \dot u\|_{L^2}^{2(p-1)/(p^2-2)}\|\dot u\|_{H^1}^{p(p-2)/(p^2-2)}\\ & \le
C\| \rho^{1/2}  \dot u\|_{L^2} +C\| \rho \dot u\|_{L^2}^{2(p-1)/(p^2-2)}\|\na \dot u\|_{L^2}^{p(p-2)/(p^2-2)}, \ea\ee
which together with (\ref{s411}), (\ref{s422}) and (\ref{pt1}) leads to
\be\la{s423} \ba
&\int_0^T \left(\| \rho \dot u\|^{1+1 /p}_{L^p}+t\|  \dot u\|^2_{H^1}\right) dt\\
&\le C+C \int_0^T\left( \| \rho^{1/2}  \dot u\|_{L^2}^2 +  t\|\na \dot u\|_{L^2}^2+ 
t^{-(p^3-p^2-2p)/(p^3-p^2-2p+2)} \right)dt \\ 
&\le C.\ea\ee
Then, it follows from (\ref{s418}) and Sobolev embedding that
\be\la{s424}\ba &\|\div u\|_{L^\infty}+\|\o\|_{L^\infty} +  \|G\|_{L^\infty} \\ 
&\le C +C \|\na G\|_{L^p} +C \|\na \o\|_{L^p}\\ &\le C +C \|\n\dot u\|_{L^p}, \ea\ee 
which combined with (\ref{p2}), (\ref{s423}) and (\ref{s424}) yields (\ref{s421})
and finishes the proof of Lemma \ref{s22}.
\end{proof}
\begin{lemma}\la{s23}
There exists a positive constant $C$ depending only on 
$T,\ q,\ \ga,\  \mu,\ \nu$, $E_0$, $\| \rho_0 \|_{L^1 \cap W^{1,q}}$ and $\| u_0 \|_{H^1}$, such that
\be\la{s431}\ba
&\sup_{0\le t\le T}\left(\norm[W^{1,q}]{ \rho}+\| u\|_{H^1}+t^{1/2}\| \n^{1/2} u_t \|_{L^2}+t^{1/2} \| u\|_{H^2} + \| \n_t \|_{L^2} \right) \\
& \quad + \int_0^T \left( \|\nabla^2 u\|^{2}_{L^2}+\|\nabla^2 u\|^{(q+1)/q}_{L^q}+t \|\nabla^2 u\|_{L^q}^2+\| \n^{1/2} u_t \|^2_{L^2}+t\| u_t\|_{H^1}^2 \right) dt\le C. \ea\ee
\end{lemma}
\begin{proof}
First, differentiating $(\ref{ns})_1$ with respect to $x$ and multiplying the resulting equation by $q |\nabla\n|^{q-2} \na \n $, we infer
\be\la{s432} \ba
& (|\nabla\n|^q)_t + \text{div}(|\nabla\n|^qu)+ (q-1)|\nabla\n|^q\text{div}u  \\
&+ q|\nabla\n|^{q-2} \p_i\n \p_i u^j \p_j\n +
q\n|\nabla\n|^{q-2}\p_i\n  \p_i \text{div}u = 0.\ea \ee 
Integrating (\ref{s432}) over $\T^2$ implies 
\be\la{s433} \ba
\frac{d}{dt} \norm[L^q]{\nabla\n}  
&\le C \norm[L^{\infty}]{\nabla u}  \norm[L^q]{\nabla\n} +C\|\na^2 u\|_{L^q} \\ 
&\le C(1+\norm[L^{\infty}]{\nabla u} ) \norm[L^q]{\nabla\n}+C \|\n\dot u\|_{L^q}, \ea\ee 
where we have used the following simple fact:
\be\la{s434} \ba
\|\na^2 u\|_{L^q} &\le C(\|\na \div u\|_{L^q}+\|\na \o\|_{L^q}) \\
&\le C (\| \nabla G\|_{L^q}+ \|\nabla P \|_{L^q})+ C\|\na \o\|_{L^q} \\
&\le C\|\n\dot u\|_{L^q} + C  \|\nabla \n \|_{L^q} , \ea \ee 
due to  (\ref{p2}) and (\ref{dc1}).

On the other hand, it follows from (\ref{gn11}) and (\ref{s418}) that
\be\la{s435}\ba &\|\div u\|_{L^\infty}+\|\o\|_{L^\infty} \\ &\le C +C \|\na G\|_{L^q}^{q/(2(q-1))} +C \|\na \o\|_{L^q}^{q/(2(q-1))}\\ &\le C +C \|\n\dot u\|_{L^q}^{q/(2(q-1))} , \ea\ee 
which together with Lemma \ref{bkm} and (\ref{s434}) yields that
\be\la{s436}\ba   \|\na u\|_{L^\infty }  
&\le C\left(\|{\rm div}u\|_{L^\infty }+ \|\o\|_{L^\infty } \right)\log(e+\|\na^2 u\|_{L^q}) +C\|\na u\|_{L^2} +C \\
&\le C\left(1+\|\n\dot u\|_{L^q}^{q/(2(q-1))}\right)\log(e+\|\rho \dot u\|_{L^q} +\|\na \rho\|_{L^q}) +C\\
&\le C\left(1+\|\n\dot u\|_{L^q} \right)\log(e+ \|\na \rho\|_{L^q}) . \ea\ee
Set
\bnn f(t):=e+\|\na \rho\|_{L^q},\quad h(t):=1+\|\n\dot u\|_{L^q}. \enn
Combining (\ref{s433}) and (\ref{s436}), we obtain
\bnn f'(t)\le Ch(t)f(t)\log f(t), \enn 
due to $f(t)\ge e$ and $h(t)\ge 1$, which yields that
\be\la{s437} (\log f(t))'\le Ch(t)\log f(t).\ee 
Subsequently, it follows from (\ref{s421}) that
\be\la{s438} \ba 
\int_0^T \| \rho \dot u\|_{L^q}^{1+1/q}  dt \le C.\ea\ee
Thus we deduce from (\ref{s437}), (\ref{s438}) and Gr\"onwall's inequality that
\be\la{s439} \ba \sup\limits_{0\le t\le T}\|\nabla
\rho\|_{L^q}\le  C, \ea\ee
which together with (\ref{s411}), (\ref{s421}), (\ref{s434}), (\ref{s438}) and (\ref{s439}) leads to
\be \la{s4310} \ba \sup_{0\le t\le T} t^{1/2} \| \na^2 u\|_{L^2} + \int_0^T \left(\|\nabla^2 u\|^{2}_{L^2}+\|\nabla^2 u\|^{(q+1)/q}_{L^q}+t \|\nabla^2 u\|_{L^q}^2 \right) dt \le C. \ea \ee

Moreover, by using $(\ref{ns})_1$, (\ref{s418}) and (\ref{s439}), we derive
\bnn  
\| \n_t\|_{L^2}\le
C\|u\|_{L^{2q/(q-2)}}\|\nabla \n \|_{L^q}+C\|\n\|_{L^\infty} \|\nabla u\|_{L^2} \le C,\enn
which implies 
\be\la{s4311} \ba  \sup_{0\le t\le T} \| \n_t\|_{L^2} \le C. \ea\ee

Finally, it follows from (\ref{s418}) and H\"older's inequality that
\be \la{s4312}\ba \int\rho|u_t|^2dx 
&\le \int\rho| \dot u |^2dx+\int \n |u\cdot\na u|^2dx \\
&\le \int\rho| \dot u |^2dx+C \| u \|_{L^4}^2 \| \na u\|_{L^4}^2 \\ 
&\le \int\rho| \dot u |^2dx+C\| \na^2 u\|_{L^2}^2.\ea \ee 
Similarly, we also have \be\la{s4313}\ba \|\nabla u_t\|_{L^2}^2 
&\le \| \nabla \dot u \|_{L^2}^2+ \| \nabla(u\cdot\nabla u)\|_{L^2}^2  \\ 
&\le \|\nabla \dot u\|_{L^2}^2+ \|u\|_{L^{2q/(q-2)}}^2\|\nabla^2u \|_{L^q}^2+ \| \nabla u \|_{L^4}^4 \\ 
&\le \|\nabla \dot u\|_{L^2}^2+C\|\nabla^2u \|_{L^q}^2+ \| \nabla u \|_{L^4}^4.\ea\ee 
Consequently, using (\ref{pt1}), (\ref{s411}), (\ref{s4310}), (\ref{s4312}) and (\ref{s4313}) yields
\be \la{s4314} \ba  \sup_{0\le t\le T} t^{1/2}\| \n^{1/2} u_t \|_{L^2} + \int_0^T \| \n^{1/2} u_t \|^2_{L^2}+ t\|u_t\|_{H^1}^2 dt \le C. \ea\ee

Combining (\ref{s418}), (\ref{s439}), (\ref{s4310}), (\ref{s4311}) and (\ref{s4314}) leads to (\ref{s431}) and hence we finish the proof of Lemma \ref{s23}.
\end{proof}

From now on, we assume that the initial data $(\n_0,u_0)$ satisfy (\ref{csol1}) and the compatibility condition (\ref{csol2}).
\begin{lemma}\la{c21}
There exists a positive constant $C$ depending only on  
$T,\ \ga,\ \mu,\ \nu,\ E_0$, $\| \rho_0 \|_{L^1 \cap L^\infty }$,
$\| u_0 \|_{H^1}$ 
and $\|g_2\|_{L^2}$, such that
\be\la{c411} \ba
\sup_{0\le t\le T}
\int\n|\dot u|^2dx
+\int_0^{ T}\int|\na\dot u|^2dxdt\le C.
\ea\ee

\end{lemma}
\begin{proof}
Taking into account the compatibility condition (\ref{csol2}), we can define
\be\la{c413} \ba
\sqrt{\n} \dot{u}(x,t=0)=g_2(x).
\ea \ee
By integrating (\ref{s417}) over $(0,T)$ and then combining it with (\ref{s418}), (\ref{s419}), we obtain (\ref{c411}).
\end{proof}
\begin{lemma}\la{c22}
There exists a positive constant $C$ depending only on 
$T,\ q,\ \ga,\ \mu,\ \nu,\ E_0$, $\| \rho_0 \|_{L^1 \cap W^{1,q}}$,
$\| u_0 \|_{H^2}$ and $\| g_2 \|_{L^2}$ such that
\be\la{c421} \ba
\sup_{0\le t\le T}\left(\norm[W^{1,q}]{ \rho} + \| u\|_{H^2} +\| \n^{1/2} u_t\|_{L^2}+\| \n_t\|_{L^2} \right) + \int_0^T \left( \|\nabla^2 u\|_{L^q}^2 + \| \na u_t\|^2_{L^2} \right) dt\le C.
\ea\ee
\end{lemma}
\begin{proof}
According to (\ref{c411}), similar to the proof of Lemma \ref{s22}, we derive (\ref{c421}). 		
\end{proof}
\begin{lemma}\la{c23}
There exists a positive constant $C$ depending only on 
$T,\ \ga,\ \mu,\ \nu,\ E_0,$ $ \| \rho_0 \|_{L^1 \cap H^2}$, $\| P(\rho_0) \|_{H^2}$,
$\| u_0 \|_{H^2}$ and $\|g_2\|_{L^2}$ such that	
\be\la{c431} \ba 
& \sup_{t\in[0,T]} \left(\norm[H^2]{\rho} + \norm[H^2]{P(\rho)}+ \|\n_t\|_{H^1}+\|P_t\|_{H^1} \right) \\
&\quad + \int_0^T\left(\| \na^3 u\|^2_{L^2}+\|\n_{tt}\|_{L^2}^2 +\|P_{tt}\|_{L^2}^2\right)dt \le C.\ea \ee
\end{lemma}
\begin{proof}
First, by combining $(\ref{ns})_1$ and following a calculation similar to that of (\ref{s432}), we get
\be\la{c432} \ba
\frac{d}{dt} & \left(\norm[L^2]{\nabla^2 \rho} + \norm[L^2]{\nabla^2P } \right) \\
&\le C\norm[L^{\infty}]{\nabla u} \left(\norm[L^2]{\nabla^2P }+\norm[L^2]{\nabla^2 \rho}\right) + C\norm[L^2]{\nabla^3 u }. \ea\ee
On the other hand, by using (\ref{p2}), (\ref{dc1}), (\ref{c421}), 
(\ref{c411}), (\ref{pt1}) and H\"older's inequality we derive
\be\la{c433} \ba
\|\nabla^3 u\|_{L^2}&\le C\left( \| \na^2\div u\|_{L^2}+ \|\na^2 \o\|_{L^2}  \right)\\
&\le C\left( \| \na^2 G \|_{L^2}+ \| \na^2 P \|_{L^2} + \|\na^2 \o\|_{L^2} \right) \\
& \le C\left(\norm[L^2]{\nabla(\rho\dot{u})} + \| \na^2 P \|_{L^2} \right) \\ 
& \le C\left(\norm[L^q]{\nabla \rho} \norm[L^{\frac{2q}{q-2}}]{\dot{u}} +\| \n \|_{L^\infty}\norm[L^2]{\nabla\dot{u}} + \| \na^2 P \|_{L^2} \right) \\
& \le C \left(1+\norm[L^2]{\nabla\dot{u}}+ \| \na^2 P \|_{L^2} \right),\ea\ee
which together with (\ref{c411}), (\ref{c421}), (\ref{c432}), (\ref{c433}) and Gr\"onwall's inequality, leads to 
\be\la{c434} \ba \sup_{t\in[0,T]} {\left(\norm[L^2]{\nabla^2P } +\norm[L^2]{\nabla^2 \rho } \right)}\le C. \ea\ee

Consequently, it follows from (\ref{c411}), (\ref{c433}) and (\ref{c434}) that
\be\la{c4341}\ba
\int_0^T  \| \na^3 u\|^2_{L^2} \le C. \ea\ee

Additionally, differentiating $(\ref{ns})_1$ with respect to $x$ yields
\bnn
\p_i \n_t+u_j\cdot \p_j \p_i \n+\p_i u_j\cdot\p_j \n+\p_i \n \div u+ \n  \p_i \div u=0,
\enn 
which together with	(\ref{c421}) and (\ref{c434}), shows that
\be\la{c435} \ba 
\|\nabla \n_t\|_{L^2}& \le C\left( \|u\|_{L^\infty}\| \nabla^2
\n\|_{L^2}+\|\nabla u\|_{L^4}\|\nabla \n\|_{L^4}+\|\n\|_{L^{\infty}}\| \nabla^2
u\|_{L^2} \right)
\le C.\ea \ee 
Combining (\ref{c421}) and (\ref{c435}) leads to
\be\la{c436} \ba
\sup_{0\le t\le T}\|\n_t\|_{H^1}\le C.\ea \ee	

Finally, differentiating $(\ref{ns})_1$ with respect to $t$ gives
\be\la{c437}\ba 
\n_{tt} +\n_t \div u +\n \div u_t + u_t\cdot\nabla \n + u\cdot\nabla \n_t = 0,\ea \ee
which together with (\ref{c437}), (\ref{c421}) and (\ref{c436}) implies
\be\la{c439} \ba &\int_0^T \|\n_{tt}\|_{L^2}^2dt \\ & \quad \le C\int_0^T
\left( \|\n_t\|_{L^4}\|\nabla u\|_{L^4}+\|\nabla
u_t\|_{L^2}+\|u_t\|_{L^4}\|\nabla \n \|_{L^4}+ \|\nabla
\n_t\|_{L^2} \right)^2dt\\& \quad \le C.\ea\ee
In addition, we can make similar arguments for $P_t$ and $P_{tt}$ by using $(\ref{ns})_1$, which yields
\be\la{c4310} \ba \sup_{0\le t\le T}\|P_t\|_{H^1}+\int_0^T \|P_{tt}\|_{L^2}^2dt \le C.\ea\ee
This combined with (\ref{c434}), (\ref{c4341}), (\ref{c436}), (\ref{c439}) and (\ref{c4310}) gives (\ref{c431}), and hence we finish the proof of Lemma \ref{c23}.	
\end{proof}
\begin{lemma}\la{c24}
There exists a positive constant $C$ depending only on 
$T,\ \ga,\ \mu,\ \nu,\ E_0$, $\| \rho_0 \|_{L^1 \cap H^2}$, 
$\| P(\rho_0) \|_{H^2}$, $\| u_0 \|_{H^2}$ and $\| g_2 \|_{L^2}$ such that
\be\la{c442}\ba
\sup\limits_{0\le t\le T} t^{1/2} \left(\| \na u_t\|_{L^2}+\| \na^3 u\|_{L^2} \right) 
+ \int_0^T t\left(\|\n^{1/2}u_{tt}\|^2_{L^2}+\| \na^2 u_t\|^2_{L^2}\right)dt\le C.
\ea\ee
\end{lemma}
\begin{proof}
First, differentiating  $(\ref{ns})_2$  with
respect to $t$ and multiplying the resulting equation  by
$u_{tt}$, and after integration by parts we can derive \be\la{c447} \ba
& \frac{1}{2}\frac{d}{dt}\int \left(\mu|\nabla u_t|^2 + (\lambda +
\mu)({\rm div}u_t)^2\right)dx+\int_{ }\rho u_{tt}^2dx \\
&=\frac{d}{dt}\left(-\frac{1}{2}\int_{ }\rho_t |u_t|^2 dx- \int_{
}\rho_t u\cdot\nabla u\cdot u_tdx+ \int_{ }P_t {\div}u_tdx\right)\\
&\quad - \frac{1}{2}\int_{ } \div(\rho u)_t |u_t|^2 dx +
\int_{ } \left( \rho_{t} u\cdot\nabla u \right)_t\cdot u_tdx \\
&\quad-\int_{ } \left(\rho u_t\cdot\nabla u + \rho u\cdot\nabla u_t \right) \cdot u_{tt}dx -\int_{ }P_{tt}{\rm div}u_tdx \\ &\triangleq
\frac{d}{dt}J_0+ \sum\limits_{k=1}^4 J_k. \ea \ee		
Next, we estimate each $J_k$, $k=0,1,\dots ,4$.

It follows from $(\ref{ns})_1$, (\ref{c421}), (\ref{c431}) and Young's inequality that
\be \la{c448} \ba |J_0|& =\left|-\frac{1}{2}\int_{
}\rho_t |u_t|^2 dx- \int_{ }\rho_t u\cdot\nabla u\cdot u_tdx+
\int_{ }P_t {\rm div}u_tdx\right|\\ &\le \left|\int_{ } {\div}(\n
u)|u_t|^2dx\right| + C\norm[L^4]{\rho_t}\norm[L^4]{u}\norm[L^4]{\nabla u}
\norm[L^4]{u_t} + C\|P_t\|_{L^2}\|\nabla u_t\|_{L^2}\\ &\le C \int_{
} \n |u||u_t||\nabla u_t| dx +C\|\nabla u_t\|_{L^2} +C \\ &\le C
\|\n^{1/2} u_t\|_{L^2}^{1/2}\|u_t\|_{L^2}^{1/2}\|\nabla
u_t\|_{L^2} +C\|\nabla u_t\|_{L^2} +C \\ &\le \varepsilon \|\nabla
u_t\|_{L^2}^2+C_\varepsilon. \ea\ee

Next, based on 
$(\ref{ns})_1$, (\ref{c421}), (\ref{c431}) and (\ref{pt1}), we can get	
\be \la{c449}\ba
|J_1|& = \left|\int_{ } \left(\rho_tu + \rho u_t \right )\cdot\nabla \left(|u_t|^2 \right)dx\right|\\
& \le C \| \na u_t\|_{L^2} \| \n_t\|_{L^6} \|u\|_{L^6} \| u_t\|_{L^6}+C \| \na u_t\|_{L^2} \| u_t\|^{2}_{L^4} \\
& \le C\norm[L^2]{\nabla u_t}+C\norm[L^2]{\nabla u_t}^2 +C\norm[L^2]{\nabla u_t}^{3}\\
& \le C\norm[L^2]{\nabla u_t}^4 +C. \ea \ee
Similarly,
\be \la{c4410}\ba
|J_2|&=\left|\int_{ }\left(\rho_t u\cdot\nabla u \right)_t\cdot u_{t}dx
\right|\\
& = \left|  \int_{ }\left(\rho_{tt} u\cdot\nabla u\cdot u_t +\rho_t
u_t\cdot\nabla u\cdot u_t+\rho_t u\cdot\nabla u_t\cdot
u_t\right)dx\right|\\ &\le   \norm[L^2]{\rho_{tt}}\norm[L^6]{u}
\norm[L^6]{\nabla u}\norm[L^6]{u_t}+\norm[L^4]{\rho_t}\norm[L^4]{u_t}^2
\norm[L^4]{\nabla u} \\
&\quad+\norm[L^2]{\nabla u_t}\norm[L^6]{\rho_t}\norm[L^6]{u}\norm[L^6]{u_t}\\
& \le C\norm[L^2]{\rho_{tt}}^2 + C\norm[L^2]{\nabla u_t}^2+C. \ea \ee

In addition, we conclude from (\ref{c421}) and Cauchy's inequality that
\be\la{c4411} \ba |J_3| & \le  \int_{ } \left( \left| \rho u_t\cdot\nabla
u\cdot u_{tt} \right|+ \left| \rho u\cdot\nabla u_t\cdot u_{tt} \right| \right) dx\\
& \le C\|\n^{1/2}u_{tt}\|_{L^2} \|u_t\|_{L^4}\|\na u\|_{L^4}+
C\|\n^{1/2}u_{tt}\|_{L^2} \|u\|_{L^\infty}\|\na u_t\|_{L^2} \\
& \le  \varepsilon  \norm[L^2]{\rho^{{1/2}}u_{tt}}^2 + C_\varepsilon \norm[L^2]{\nabla u_t}^2 +C_\varepsilon. \ea\ee

Finally, direct calculation yields that
\be\la{c4412} \ba |J_4|=
\left|\int_{ }P_{tt}{\rm div}u_tdx\right| & \le
\norm[L^2]{P_{tt}}\norm[L^2]{{\rm div}u_t}\\& \le
C\norm[L^2]{P_{tt}}^2 + C\norm[L^2]{\nabla u_t}^2. \ea\ee
Substituting (\ref{c448})--(\ref{c4412}) into (\ref{c447}) results in
\be\la{c4413} \ba
& \frac{1}{2}\frac{d}{dt}\int \left(\mu|\nabla u_t|^2 + (\lambda +
\mu)({\rm div}u_t)^2\right)dx+\int_{ }\rho u_{tt}^2dx
\\
& \quad \le \frac{d}{dt}J_0+\varepsilon \norm[L^2]{\rho^{{1/2}}u_{tt}}^2 +
C_\varepsilon \left(1+\norm[L^2]{\nabla u_t}^4 +\norm[L^2]{\n_{tt}}^2+\norm[L^2]{P_{tt}}^2 \right).	
\ea \ee	
Multiplying (\ref{c4413})by $t$ and integrating the resulting equation over $(0,T)$, 
and then taking $\varepsilon$ suitably small,
we obtain from (\ref{c4413}), (\ref{c421}), (\ref{c431}) and Gr\"onwall's inequality that
\be\la{c4414} \ba \sup\limits_{0\le t\le T}t\int_{ }|\nabla u_t|^2dx
+ \int_0^Tt\int\rho u_{tt}^2dxdt \le C. \ea\ee
On the other hand, it follows from  (\ref{c433}) and (\ref{c421}) that 
\be \la{c4415} \ba 
\|\na^3 u \|_{L^2} &\le C\left(1+\norm[L^2]{\nabla\dot{u}}+ \| \na^2 P \|_{L^2} \right)\\
&\le C \left(1+\| \nabla u_t \|_{L^2}+\| \nabla u\|^2_{L^4}+\| u \|_{L^\infty}\| \nabla^2 u  \|_{L^2} \right) \\
&\le C\| \nabla   u_t  \|_{L^2} + C,\ea \ee
which together with (\ref{c4414}) leads to
\be \la{c4416} \ba 
\sup_{0\le t\le T} \sqrt{t} \|\na^3 u \|_{L^2} \le C.
\ea \ee

Additionally, we deduce from (\ref{p2}), (\ref{dc1}), (\ref{c421}) and (\ref{c431}) that
\be \la{c4417} \ba 
\|\na^2 u_t \|_{L^2} &\le C \left(\| \na \div u_t \|_{L^2}+\| \na\o_t \|_{L^2} \right) \\
& \le C\left(\| \na G_t \|_{L^2}+\| \na P_t \|_{L^2}+\| \na \o_t \|_{L^2} \right) \\
& \le C \|  \left(\n \dot u\right)_t \|_{L^2} +C \\
& \le C\left(\|\n  u_{tt}\|_{L^2}+ \|\n_t\|_{L^4}
\|u_t\|_{L^4}+\|\n_t\|_{L^4}\| u\|_{L^\infty}\|\nabla
u\|_{L^4}\right)\\&\quad
+C\left(\| u_t\|_{L^4}\|\nabla u\|_{L^4}+ \| u\|_{L^\infty}
\|\nabla u_t\|_{L^2}\right)+C \\ 
&\le C\| \na u_t \|_{L^2} + C\|\n  u_{tt}\|_{L^2} +C, \ea \ee
which together with (\ref{c4414}), implies 
\be\la{c4418} \ba
\int_{0}^T t \| \nabla ^2 u_t\|_{L^2}^2dt\le C . \ea \ee
Consequently, combining (\ref{c4414}), (\ref{c4416}) and (\ref{c4418}) leads to 
(\ref{c442}), which completes the proof of Lemma \ref{c24}.
\end{proof}
\begin{lemma}\la{c25}
There exists a positive constant $C$ depending only on
$T,\ \ga,\ \mu,\ \nu,\ E_0$, $\| \rho_0 \|_{L^1 \cap W^{2,q}}$,
$\| P(\rho_0) \|_{W^{2,q}}$, $\| u_0 \|_{H^2}$ 
and $\| g_2 \|_{L^2}$ such that
\be\la{c451} \ba
\sup_{0\leq t\leq T}\left(\|\nabla^2 \n\|_{L^q } +\|\nabla^2 P  \|_{L^q }\right) \leq C,
\ea\ee
\be\la{c452} \ba
\sup_{0\leq t\leq T} t^{1/2} \|\na^2(\n u)\|_{L^q}
+\int_{0}^T \left(\|\na u_t \|_{L^q}^{1+1/q}+\|\na^3 u \|_{L^q}^{1+1/q}+ t \|\na^3 u \|_{L^q}^2 \right) dt \le C .\ea\ee
\end{lemma}
\begin{proof}
First, similar to the calculation of (\ref{s432}) and based on $(\ref{ns})_1$, we can derive
\be\la{c453} \ba \frac{d}{dt} \|\na^2 P\|_{L^q}
& \le C\left(\||\na^2 u||\na P|\|_{L^q}+ \||\na u||\na^2 P|\|_{L^q}+\| \na^3u
\|_{L^q}\right)\\
& \le C \left( \| \na^3 u \|_{L^2} \|\na^2P\|_{L^2} +\| \na u\|_{L^\infty}\|\na^2 P \|_{L^q} +\| \na^3u\|_{L^q} \right)\\
& \le C \left(\| \na u\|_{L^\infty}\|\na^2 P \|_{L^q} +\| \na^3u\|_{L^q} \right).\ea\ee

Next, similar to (\ref{c433}), it follows from (\ref{p2}), (\ref{dc1}), (\ref{c411}) and (\ref{c421}) that
\be\la{c454}\ba
\|\na^3 u \|_{L^q} 
& \le C\left(\norm[L^q]{\nabla(\rho\dot{u})} + \| \na^2 P \|_{L^q} \right) \\
& \le C\left(1+\norm[L^q]{\nabla\dot{u}}+ \| \na^2 P \|_{L^q} \right) \\
& \le C\left(1+\| \nabla u_t \|_{L^q}+\| \na u \|^2_{L^{2q}}+ \|u\|_{L^\infty}\| \na^2 u\|_{L^q} + \| \na^2 P \|_{L^q} \right)\\
& \le  C\left(1+\| \nabla   u_t  \|_{L^q} + \|\na^2 u \|_{L^2}^{\frac{q}{2(q-1)}} \|\na^3 u \|_{L^q}^{\frac{q-2}{2(q-1)}} + \| \na^2 P \|_{L^q} \right) \\
& \le \frac{1}{2} \|\na^3 u \|_{L^q}+ C\| \nabla   u_t  \|_{L^q} +  C \| \na^2 P \|_{L^q}+C, \ea \ee
which implies
\be\la{c455} \ba
\|\na^3 u \|_{L^q}  \le  C\| \nabla   u_t  \|_{L^q} +  C \| \na^2 P \|_{L^q}  +C.
\ea \ee
On the other hand, it follows from (\ref{gn11}) and (\ref{c442}) that
\be\la{c456}\ba 
\int_0^{T} \|\na u_t \|_{L^q}^{1+1/q}dt & \le C\int_0^{T}\left( (t\|\na   u_t\|_{L^2}^2)^{1/q}(t\|\na^2  u_t\|_{L^2}^2)^{(q-2)/(2q) }t^{-1/2}\right)^{1+1/q}dt \\ 
& \le C \int_0^{T_0}\left( t\|\na^2 u_t\|_{L^2}^2 + t^{-(q^2+q)/(q^2+q+2)}\right)dt  \\ &\le  C, \ea\ee
which together with (\ref{c453}), (\ref{c456}), (\ref{c421}) and Gr\"onwall's inequality yields
\be\la{c457} \ba 
\sup\limits_{0\le t\le T}\|\nabla^2  P\|_{L^q} \le C.\ea \ee
Similarly, by using $(\ref{ns})_1$ we derive
\be\la{c458} \ba
\sup\limits_{0\le t\le T}\| \na ^2 \n \|_{L^q} \le C,\ea \ee
which along with (\ref{c457}) and (\ref{c458}) gives (\ref{c451}). 

Therefore, we deduce from (\ref{c458}), (\ref{c421}) and H\"older's inequality that 
\be\la{c459}\ba 
\|\na^2(\n u)\|_{L^q} & \le C  \||\na^2 \n|| u |\|_{L^q} 
+C \| \na u \|_{L^q}+C \|\na^2  u \|_{L^q} \\
& \le C \|\na^2\n\|_{L^q} \| u\|_{L^\infty} + C\|\na^3  u\|_{L^{2}} \\
& \le C+ C\|\na^3  u\|_{L^{2}}, \ea \ee
which together with (\ref{c442}) yields
\be\la{c4510}\ba
\sup\limits_{0\le t\le T} \sqrt{t} \|\na^2(\n u)\|_{L^q} \le C. \ea \ee 

It follows from (\ref{c442}), (\ref{c455}), (\ref{c456}) and (\ref{c457}) that
\be\la{c4511} \ba
\int_0^T \|\na^3 u \|_{L^q}^{1+1/q}+  t \|\na^3 u \|^2_{L^q}dt\le C,\ea \ee
which together with (\ref{c456}), (\ref{c4510}) and (\ref{c4511}) gives (\ref{c452}), and hence we finish the proof of Lemma \ref{c25}.
\end{proof}
\begin{lemma}\la{c26}
There exists a positive constant C depending only on $T,\ \mu,\ \nu,\ \ga,\ E_0$,
$\| \rho_0 \|_{L^1 \cap W^{2,q}}$, $\| P(\rho_0) \|_{W^{2,q}}$,
$\| u_0 \|_{H^2}$ and $\| g_2 \|_{L^2}$ such that
\be\la{c461} \ba
\sup_{0\leq t\leq T} t\left(\|\n^{1/2}u_{tt}\|_{L^2}+  \|\na^3 u \|_{L^q} + \|\na^2
u_t \|_{L^2}  \right) +\int_{0}^T  t^2 \|\nabla u_{tt}\|_{L^2}^2 dt\leq C .\ea\ee	
\end{lemma}
\begin{proof}
Differentiating $(\ref{ns})_2$ with respect to $t$ twice enables us to obtain
\be\la{c462}\ba &\n u_{ttt}+\n u\cdot\na u_{tt}-\mu\Delta
u_{tt}-(\mu+\lambda)\nabla{\rm div}u_{tt}\\&= -2\n_t u_{tt}
-\n_{tt} u_t-2(\n u)_t\cdot\na u_t -\left(\n u \right)_{tt} \cdot\na u -\na P_{tt}.
\ea\ee
Multiplying (\ref{c462}) by $u_{tt}$ and integrating the
resulting equation over ${\T^2}$, after integration by parts we have	
\be \la{c463}\ba &\frac{1}{2}\frac{d}{dt}\int_{ }\n
|u_{tt}|^2dx+\int_{ }\left(\mu|\na u_{tt}|^2+(\mu+\lambda)({\div}u_{tt})^2\right)dx \\ &=-4\int_{ }  u^i_{tt}\n u\cdot\na
u^i_{tt} dx-\int_{ } \left[\na (u_t\cdot u_{tt})+2\na
u_t\cdot u_{tt}\right] \cdot (\n u)_t  dx\\
&\quad -\int_{} \left(\n_{tt}u+2\n_tu_t \right) \cdot\na u\cdot u_{tt}dx 
-\int_{} \left( \n u_{tt}\cdot\na u\cdot u_{tt}-P_{tt}{\div}u_{tt} \right) dx\\ 
& \triangleq\sum_{i=1}^4I_i,\ea\ee 
due to $(\ref{ns})_1$. 

Next, we estimate each $I_i(i=1,\cdots,4)$ as follows:

First, using H\"older's and Cauchy's inequalities gives \be \la{c464} \ba |I_1|&\le
C\|\n^{1/2}u_{tt}\|_{L^2}\|\na u_{tt}\|_{L^2}\| u \|_{L^\infty}\\
&\le \varepsilon  \|\na u_{tt}\|_{L^2}^2 + C_\varepsilon \|\n^{1/2}u_{tt}\|^2_{L^2} .\ea\ee 
Then, we deduce from (\ref{c421}), (\ref{c431}), (\ref{pt1})
and Cauchy's inequality that 
\be \la{c465}\ba |I_2|&\le C \left(\| u_{tt}\|_{L^6}\| \na
u_t\|_{L^2}+\| \na u_{tt}\|_{L^2}\| u_t\|_{L^6}\right) \left(\|\n
u_t\|_{L^3}+\|\n_t u\|_{L^3}\right) \\
&\le C \left (\|\n^{1/2}u_{tt}\|_{L^2} + \| \na u_{tt}\|_{L^2} \right) \left(1+\|\na u_{t}\|_{L^2} \right) \left( 1 + \|u_t\|_{L^3} \right) \\ 
&\le \varepsilon \| \na u_{tt}\|^2_{L^2} + C_\varepsilon \|\n^{1/2}u_{tt}\|^2_{L^2} + C_\varepsilon \left(1+\|\na u_{t}\|^4_{L^2} \right).\ea\ee
In addition, it follows from (\ref{pt1}), (\ref{c421}) and (\ref{c431}) that
\be\la{c466}\ba 
|I_3|&\le C\left(\|\n_{tt}\|_{L^2} \|u\|_{L^\infty}+\|\n_{t}\|_{L^4}\|u_{t}\|_{L^4} \right)\|\na u \|_{L^4} \|u_{tt}\|_{L^4} \\
& \le C \left( \|\n_{tt}\|_{L^2}+1+ \|\na u_t \|_{L^2}  \right) \left(\|\n^{1/2}u_{tt}\|_{L^2} + \| \na u_{tt}\|_{L^2} \right) \\
&\le \varepsilon \| \na u_{tt}\|^2_{L^2} + C_\varepsilon \|\n^{1/2}u_{tt}\|^2_{L^2} + C_\varepsilon \left(1+\|\na u_{t}\|^2_{L^2} \right) +C_\varepsilon \|\n_{tt}\|^2_{L^2} .\ea\ee
Similarly, with the aid of (\ref{pt1}), (\ref{c421}) and 
Cauchy's inequality, it holds that
\be  \la{c467}\ba |I_4|&\le C\|\n^{1/2} u_{tt}\|_{L^2} \|\na
u\|_{L^4} \| u_{tt}\|_{L^4} +C \|P_{tt}\|_{L^2}\|\na
u_{tt}\|_{L^2}\\
&\le \varepsilon \| \na u_{tt}\|^2_{L^2} + C_\varepsilon \|\n^{1/2}u_{tt}\|^2_{L^2} + C_\varepsilon \| P_{tt}\|^2_{L^2}. \ea\ee	
Substituting (\ref{c464})--(\ref{c467}) into (\ref{c463})
and taking $\varepsilon$ suitably small, we conclude that
\be\la{c468}\ba
& \frac{d}{dt}\int_{ }\n |u_{tt}|^2dx+\int_{ }\left(\mu|\na u_{tt}|^2+(\mu+\lambda)({\rm
	div}u_{tt})^2\right)dx \\
& \le C \|\n^{1/2}u_{tt}\|^2_{L^2}+C\left(\|\n_{tt}\|^2_{L^2}+\|P_{tt}\|^2_{L^2} \right) +C\left(1+\|\na u_{t}\|^4_{L^2} \right).
\ea \ee	
Multiplying (\ref{c468}) by $t^2$, and integrating the resulting equation over $(0,T)$ yields
\be\la{c469}\ba
\sup_{0\le t\le T} t^2 \int_{ }\n |u_{tt}|^2dx+\int_{0}^T t^2 \int_{ }|\nabla
u_{tt}|^2dxdt\le C, \ea \ee
due to (\ref{c431}) and (\ref{c442}).

Consequently, it follows from (\ref{c455}), (\ref{c4417}), (\ref{c442}),
(\ref{c451}) and (\ref{c469}) that	
\be\la{c4610}\ba	
\sup_{0\le t\le T} t \left( \|\na^3 u \|_{L^q} + \|\na^2 u_t \|_{L^2} \right ) \le C,
\ea\ee
which together with (\ref{c469}) and (\ref{c4610}) implies (\ref{c461}),
and hence we finish the proof of Lemma \ref{c26}.
\end{proof}

\section{Proofs of Theorem \ref{th0}-\ref{th3}}
In this section, we are devoted to proving the main results.
First, we establish the global existence of the classical solution to problem (\ref{ns})--(\ref{i3}). 
Then, by approximating the initial data and applying standard compactness arguments, we can prove the global existence of weak and strong solutions.

Proof of Theorem \ref{th2}.
Let $(\n_0,u_0)$  be the initial data in Theorem \ref{th2}, 
satisfying (\ref{csol1}) and (\ref{csol2}).	
By the local existence result(Lemma \ref{lct}), there exists a $T_*>0$ such that the problem (\ref{ns})--(\ref{i3}) 
has a unique classical solution $(\n,u)$ on $\T^2 \times (0,T_*]$. 
Subsequently, we use the a priori estimates, Lemma \ref{g5},
Lemma \ref{c25} and Lemma \ref{c26}, to extend the local classical solution$(\n,u)$ to all time.

Firstly, since $\n \in C \left([0,T_*];W^{2,q} \right)$
and $\n_0 \le \|\n_0\|_{L^\infty}$,
there exists a $T_1 \in (0,T_*] $ such that (\ref{pd51}) holds for $T=T_1$.

Next, we introduce the following notation:
\be \la{y522}\ba	
T^* :=\sup\{T\,|\,(\ref{pd51}) \  \text{holds} \}.
\ea \ee	
Obviously, $T^*\ge T_1>0$. 
Furthermore, for any $0<\tau<T\leq T^*$
with $T$ finite, we can derive from Lemma \ref{c22}, Lemma \ref{c25} and Lemma \ref{c26} that
\be\la{y523}\ba 
u \in C \left([\tau ,T];C^2(\T^2) \right) ,\quad
u_t \in C\left([\tau ,T];C(\T^2) \right),
\ea\ee 
where we have used the standard embedding
$$L^\infty(\tau ,T;W^{3,q})\cap H^1(\tau ,T;H^2)\hookrightarrow
C\left([\tau ,T];C^2(\T^2) \right),  $$
and	
$$L^\infty(\tau ,T;H^2)\cap H^1(\tau ,T;L^2)\hookrightarrow
C\left([\tau ,T]; C(\T^2) \right).  $$	
Moreover, it follows from $(\ref{ns})_1$, Lemma \ref{c23}, 
Lemma \ref{c25} and Lemma \ref{c26}
as well as the standard arguments in \cite{L1} that
\be \la{y524}\ba 
\n \in C \left([0,T];W^{2,q} \right).\ea \ee 	
By combining (\ref{y523}) with (\ref{y524}), we derive
\be \la{y525}\ba 
\n^{1/2}u_t\in C([\tau,T];L^2),\ea \ee
which together with (\ref{y523}), yields
\be \la{y526}\ba 
\n^{1/2} \dot u \in C([\tau,T];L^2).\ea \ee
Finally, we claim that
\be \la{y527}\ba T^*=\infty .\ea \ee
Assume, for the sake of contradiction, that  $T^*<\infty$. Then by Lemma \ref{g5}, (\ref{pd52}) holds for
$T=T^*$. 
It follows from Lemma \ref{c24}, Lemma \ref{c25}, Lemma \ref{c26}, 
and (\ref{y523}), (\ref{y524}), (\ref{y526}) that $(\n(x,T^*),u(x,T^*))$ satisfies (\ref{csol1}) and (\ref{csol2}), where $g_2(x)\triangleq \left( \n^{1/2} \dot u \right)(x, T^*),\,\,x\in \T^2.$ 
Thus, Lemma \ref{lct} implies that there exists some $T^{**}>T^*$, such that
(\ref{pd51}) holds for $T=T^{**}$, which contradicts (\ref{y522}),
and hence (\ref{y527}) holds. 
Finally, Lemma \ref{c24}, Lemma \ref{c25} and Lemma \ref{c26} show that $(\rho,u)$
is in fact the classical solution defined on $\T^2\times(0,T]$ for any
$0<T<T^*=\infty$.
Additionally, according to Lemma \ref{g3} we obtain that $(\n,u)$ satisfies (\ref{wsol4}).
And the uniqueness of $(\n,u)$ satisfying (\ref{csol4}) is similar to the proof of \cite{Ge}, and hence the proof of Theorem \ref{th2} is finished.	

Proof of Theorem \ref{th0}.
Let $(\n_0,u_0)$  be the initial data in Theorem \ref{th0}, satisfying (\ref{wsol1}).
For a constant $ \de\in (0,1),$ we define
\be\la{y529} \ba
\n_0^\de \triangleq k_\de*\n_0 , \quad u_0^\de\triangleq k_\de*u_0,
\ea \ee
where $k_\de$ is the standard mollifying kernel of width $\de$. 
Clearly, it holds that $\n_0^\de,u_0^\de \in C^\infty$, 
and for any $1 \le p < \infty$,
\bnn
\lim\limits_{\de\rightarrow 0}\left(\|\n_0^\de-\n_0\|_{L^p}
+\|u_0^\de-u_0\|_{H^1}\right)=0,
\enn
and additionally,
\bnn
\n_0^\de \rightharpoonup \n_0  \mbox{ weakly * in } L^\infty(\T^2).
\enn

According to the proof of Theorem \ref{th2}, we know that for the problem (\ref{ns})--(\ref{i3})
in which $(\n_0,m_0)$ is replaced by $(\n_0^\de, \n_0^\de u_0^\de )$,
this problem admits a unique global classical solution  $(\n^\de,u^\de)$  satisfying Lemma
\ref{s21} and Lemma \ref{s22}, and the constant $C$ is independent of $\de$.

Therefore, it follows from (\ref{s411}), (\ref{s418}) and (\ref{s421}) that
\be\la{ws1} \ba
\sup_{0\le t\le T} \left( \| \n^\de \|_{L^\infty} + \| u^\de \|_{H^1} + \si \| \sqrt{\n^\de} \dot u^\de \|^2_{L^2} \right) + \int_0^T t \| u_t^\de \|^2_2 dt \le C,
\ea\ee
where the constant $C$ is independent of $\de$. Combining this and Aubin-Lions Lemma, without loss of generality, we conclude that
\be\la{ws2}\ba
\begin{cases}
	\n^\de \rightharpoonup \n  \mbox{ weakly * in } L^\infty(\T^2 \times (0,T)),\\
	u^\de \rightharpoonup u  \mbox{ weakly * in } \ L^\infty(0,T;H^1),\\
	u^\de \to   u  \mbox{ strongly  in } \ C( [\tau,T];L^p ),
\end{cases}
\ea\ee 
for any $0<\tau<T<\infty$ and $1 \le p < \infty$.

Moreover, by using the compactness arguments in \cite{DM,F,L2}, we get
\be\la{ws3}\ba
\n^\de \to \n  \mbox{ strongly  in } L^p(\T^2 \times (0,T)),
\ea\ee 
for any $1 \le p < \infty$.

Since $(\n^\de,u^\de)$ satisfies (\ref{ns}), using (\ref{ws2}), (\ref{ws3}) and
letting $\de\rightarrow 0$,
we can show that $(\n,u)$ is a global weak solution of the problem (\ref{ns})--(\ref{i3}).
Moreover, we deduce from (\ref{pd301}), (\ref{pd302}) and (\ref{pd303}) 
that for any $1 \le s < \infty$ and $t \ge 1$,
\be\la{ws4}\ba
\| \n^\de-\ol{\n^\de}\|^{s}_{L^s} \le C e^{-2 \alpha_0 t},\quad
\| \o^\de \|^2_{L^2} +\nu \| \div u^\de \|^2_{L^2}  \le C e^{-\alpha_0 t},\quad  \| \sqrt{\n^\de} \dot u^\de \|^2_{L^2} \le C e^{-\alpha_0 t},
\ea\ee
where the constant $C$ is independent of $\de$. Then based on (\ref{ws1}), (\ref{ws2}) and (\ref{ws3}),
we obtain that $(\n,u)$ satisfies (\ref{wsol4}), and the proof of Theorem \ref{th0} is completed.

Proof of Theorem \ref{th1}.
Let $(\n_0,u_0)$  be the initial data in Theorem \ref{th1}, satisfying (\ref{ssol1}).
For a constant $ \de\in (0,1),$ we identically define
\be\la{y529} \ba
\n_0^\de \triangleq k_\de*\n_0, \quad u_0^\de\triangleq k_\de*u_0.
\ea \ee
Obviously, we have $\n_0^\de,u_0^\de \in C^\infty$, and 
\bnn
\lim\limits_{\de\rightarrow 0}\left(\|\n_0^\de-\n_0\|_{ W^{1,q}}+\|u_0^\de-u_0\|_{H^1}\right)=0.
\enn
Based on the proof of Theorem \ref{th2}, we can infer that the problem (\ref{ns})--(\ref{i3}), with $(\n_0,u_0)$ substituted by $(\n_0^\de, u_0^\de )$,
admits a unique global classical solution $(\n^\de,u^\de)$ that conforms to Lemmas
\ref{s21}, \ref{s22} and \ref{s23}, and the constant $C$ is independent of $\de$.
By letting $\de\rightarrow 0$
and using standard arguments (see \cite{HL,LZZ}), we can conclude that the problem  (\ref{ns})--(\ref{i3}) 
has a global strong solution $(\n,u)$ satisfying (\ref{ssol4}) and (\ref{wsol4}).
The proof of uniqueness of $(\n,u)$ satisfying (\ref{ssol4}) is similar to \cite{Ge}, and we finish the proof of Theorem \ref{th1}.

Proof of Theorem \ref{th3}.
For $x_0 \in \T^2$, 
we take into account the characteristic curve $X(s)$ which initially passes through $x_0$, where $X(s)$ satisfies
\be \la{pbu1} \ba
\begin{cases}
	\frac{d}{ds}X(s) = u(X(s),s),   \\
	X(0) = x_0.
\end{cases}
\ea\ee
Since $(\n,u)$ is a strong solution that satisfies (\ref{ssol4}), $X(s)$ is well-defined. 
Moreover, from $(\ref{ns})_1$, it holds that 
\be \la{pbu2} \ba
\n \left(X(t),t\right) =\n_0(x_0) \exp \left \{ -\int_0^{t} \div u\left(X(\tau),\tau \right) d\tau \right \} \quad \mbox{  for all }  t\ge 0.
\ea\ee
By the assumption that $\n_0(x_0) = 0$, 
it follows that $\n \left(X(t),t\right) \equiv 0 \mbox{  for all }  t\ge 0$.
This implies that
\bnn \|\rho (x,t)-\ol{\n_{0}} \|_{C\left( \T^2 \right)} \ge
\left|\n \left(X(t),t\right)-\ol{\n_{0}} \right|  \equiv \ol{\n_{0}}>0,\enn 
which together with (\ref{gn11}) gives 
\be \la{pbu3} \ba
\ol{\n_{0}} \le \|\rho (x,t)-\ol{\n_{0}} \|_{C\left( \T^2  \right)}
&\le C \| \nabla \rho  (x,t) \|^a_{L^r}
\|\rho  (x,t)-\ol{\n_{0}} \|^{1-a}_{L^2},
\ea \ee
with $a=r/(2(r-1))\in  (\frac{1}{2},1)$.
Combining (\ref{pbu3}) and (\ref{wsol4}) leads to (\ref{pbu0}), and hence we finish the proof of Theorem \ref{th3}.

Proof of Theorem \ref{th01}.
For any $\nu \ge \nu_1$, we deduce from (\ref{pt1}), (\ref{0x1}), (\ref{0x10}) and (\ref{pd51})
that $\{\n^{\nu}\}_\nu$ is bounded in $L^\infty(\T^2 \times (0,\infty))$, and for any $0<\tau<T<\infty$,
$\{u^{\nu}\}_\nu$ is bounded in $L^\infty(\tau,T;H^1)\cap L^2(0,T;H^1)$.
Moreover, by combining (\ref{pt1}) with (\ref{gw}) and applying H\"older's inequality, we derive
\be\la{ins01}\ba
\| u^{\nu}_t \|_{L^2}
& \le C \left( \| \dot{u^{\nu}} \|_{L^2} + \| u^{\nu} \cdot \na u^{\nu} \|_{L^2} \right) \\
& \le C \left( \| \sqrt{\n^{\nu}} \dot{u^{\nu}} \|_{L^2} + \| \na \dot{u^{\nu}} \|_{L^2}
+ \| u^{\nu} \|_{L^4} \| \na u^{\nu} \|_{L^4} \right) \\
& \le C \| \sqrt{\n^{\nu}} \dot{u^{\nu}} \|_{L^2} + C \| \na \dot{u^{\nu}} \|_{L^2}
+ C \| u^{\nu} \|_{L^4} \left( \| \n^\nu \dot{u^{\nu}} \|_{L^2} +\| P^\nu -\ol{P^\nu} \|_{L^4} \right),
\ea\ee
which together with (\ref{0x10}) and (\ref{pd11}) implies $\{u^{\nu}\}_\nu$ is bounded in $H^1(\tau,T;L^2)$.

Therefore, with a slight abuse of notation, 
there exists a subsequence $(\n^n,u^n)$ of $(\n^{\nu},u^{\nu})$ and
$\n \in L^\infty(\T^2 \times (0,\infty)), u \in L^\infty(\tau,T;H^1)\cap L^2(0,T;H^1)$ such that
\be\la{ins1}\ba
\begin{cases}
\n^n \rightharpoonup \n  \mbox{ weakly* in } L^\infty(\T^2 \times (0,\infty)),\\
u^n \rightharpoonup u  \mbox{ weakly* in } \ L^\infty(\tau,T;H^1)\cap L^2(0,T;H^1), \\
u^n \to u  \mbox{ strongly  in } \ L^\infty(\tau,T;L^p),
\end{cases}
\ea\ee 
for any $1 \le p < \infty$ and $0<T<\infty$.

Then, we set $G^n:=n\div u^n-(P^n-\ol{P^n})$ and $\o^n:= \na^\bot \cdot u^n$.
Based on (\ref{p2}), (\ref{wsol2}),
(\ref{0x10}) and Poincar\'e's inequality, 
we can conclude that $\{G^n\}_n$ and $\{\o^n\}_n$ are bounded in $L^2(\tau,\infty;H^1)$.
Hence, without loss of generality, we can assume that there exists $\pi \in L^2(\tau,\infty;H^1)$ such that
\be\la{ins3}\ba
G^n \rightharpoonup - \pi \quad  \mbox{ weakly in } L^2(\tau,\infty;H^1).
\ea\ee
Rewriting $(\ref{ns})_2$, we derive that $(\n^n,u^n)$ satisfies
\be\la{ins4}\ba
(\n^n u^n)_t+\div(\n^n u^n\otimes u^n) -\na G^n -\mu \na^{\bot} w^n =0.
\ea\ee
Passing to the limit as $n \to \infty$ implies that $(\n,u)$ satisfies
\be\la{ins6}\ba
\begin{cases}
\n_t+\div(\n u)=0,\\
(\n u)_t+\div(\n u\otimes u) -\mu \na^{\bot} w + \na \pi =0.
\end{cases} 
\ea\ee
On the other hand, based on (\ref{0x1}) and (\ref{0x10}), we conclude that for any $0<\tau<\infty$
\be\la{ins7}\ba
\div u^{\nu} = O(\nu^{-1/2}) \  in \  L^2(\T^2 \times (0,\infty)) \cap L^\infty(\tau,\infty;L^2),
\ea\ee
which yields (\ref{isol4}).
Then, from (\ref{ins1}) and (\ref{ins7}), we deduce that $\div u=0$.
Combining this with the equality $\Delta u = \na \div u + \na^{\bot} w $, leads to $\na^{\bot} w=\Delta u$.
Consequently, the limiting solution $(\n,u)$ satisfies (\ref{isol2}) and (\ref{isol3}).

Moreover, when $\nu \ge \nu_1$ and the initial data $(\n_0,u_0)$ satisfies (\ref{ws}),
the a priori estimates (\ref{pt1}), (\ref{0x1}), (\ref{0x10}), (\ref{pd11}) and (\ref{pd51})
yield that the sequence $\{\n^{\nu}\}_\nu$ is bounded in $L^\infty(\T^2 \times (0,\infty))$,
Furthermore, for any $0<T<\infty$, the sequence
$\{u^{\nu}\}_\nu$ is bounded in $L^\infty(0,T;H^1)$,
$\{\na u^{\nu}\}_\nu$ is bounded in $L^2(\T^2 \times (0,\infty)) \cap L^\infty(0,\infty;L^2)$,
and $\{\sqrt{\n^\nu} \du^{\nu}\}_\nu$ is bounded in $L^2(0,\infty;L^2)$.
Then, multiplying (\ref{pd118}) by $\si$, integrating it over $(0,T)$, \
and applying (\ref{pd120}) and (\ref{x1000}) implies that
$\{\sqrt{t} \sqrt{\n^\nu} \dot{u^{\nu}}\}_\nu$ is bounded in $L^\infty(0,T;L^2)$
and $\{\sqrt{t} \na \dot{u^{\nu}}\}_\nu$ is bounded in $L^2(\T^2 \times (0,T))$.
In addition, from (\ref{p2}) we obtain that the sequences
$\{\sqrt{t} \na G^{\nu}\}_\nu$ and $\{\sqrt{t} \na \o^{\nu}\}_\nu$ are bounded in $L^\infty(0,T;L^2) \cap L^2(0,T;L^p)$.
Therefore, by arguments similar to those above,
we conclude that there exists a subsequence of $(\n^{\nu},u^{\nu})$
converging to a global solution of (\ref{isol2}) satisfying (\ref{lws1}).
Then, according to the result in \cite{DM2}, 
the system (\ref{isol2}) with initial data $(\n_0,u_0)$ satisfying (\ref{ws}) has a unique global solution.
This implies that the whole sequence $(\n^{\nu},u^{\nu})$ converges to the global solution of (\ref{isol2}), and $(\n,u)$ satisfies (\ref{lws1}).

Finally, for initial data $(\n_0,u_0)$ that satisfy the regularity conditions (\ref{insc1}) and the compatibility condition (\ref{insc2}),
Corollary 1.4 in \cite{HW} guarantees that the system (\ref{isol2}) admits 
a unique global strong solution $(\n,u)$ satisfying (\ref{insc3}).
Since \cite{DM2} establishes the uniqueness of solutions to (\ref{isol2}) under the conditions (\ref{lws1}),
we conclude that the entire sequence $(\n^{\nu},u^{\nu})$ converges to the 
unique global strong solution of (\ref{isol2}) and $(\n,u)$ satisfies (\ref{insc3}).
Thereby, this completes the proof of Theorem \ref{th01}.

\end{document}